\documentclass[12pt]{amsart}
\usepackage{}
\usepackage{amsmath}
\usepackage{amsfonts}
\usepackage{amssymb}
\usepackage[all,cmtip]{xy}           
\usepackage{xspace}
\usepackage{bbding}
\usepackage{txfonts}
\usepackage[shortlabels]{enumitem}
\usepackage{ifpdf}
\ifpdf
\usepackage[colorlinks,final,backref=page,hyperindex]{hyperref}
\else
\usepackage[colorlinks,final,backref=page,hyperindex,hypertex]{hyperref}
\fi
\usepackage{tikz}
\usepackage[active]{srcltx}

\makeatletter

\newtheorem{theorem}{Theorem}[section]
\newtheorem{prop}[theorem]{Proposition}
\newtheorem{lemma}[theorem]{Lemma}
\newtheorem{coro}[theorem]{Corollary}
\newtheorem{prop-def}{Proposition-Definition}[section]

\theoremstyle{definition}
\newtheorem{defn}[theorem]{Definition}

\newtheorem{remark}[theorem]{Remark}
\newtheorem{exam}[theorem]{Example}

\newcommand{\nc}{\newcommand}

\newcommand {\emptycomment}[1]{}

\nc{\delete}[1]{{}}
\nc{\mmargin}[1]{}

\nc{\mlabel}[1]{\label{#1}}  
\nc{\mcite}[1]{\cite{#1}}  
\nc{\mref}[1]{\ref{#1}}  
\nc{\meqref}[1]{\eqref{#1}}  
\nc{\mbibitem}[1]{\bibitem{#1}} 

\delete{
	\nc{\mlabel}[1]{\label{#1}  
		{\hfill \hspace{1cm}{\bf{{\ }\hfill(#1)}}}}
	\nc{\mcite}[1]{\cite{#1}{{\bf{{\ }(#1)}}}}  
	\nc{\mref}[1]{\ref{#1}{{\bf{{\ }(#1)}}}}  
	\nc{\meqref}[1]{\eqref{#1}{{\bf{{\ }(#1)}}}}  
	\nc{\mbibitem}[1]{\bibitem[\bf #1]{#1}} 
}

\newcommand{\g}{\mathfrak g}
\newcommand{\h}{\mathfrak h}

 \font\cyrs=wncyr7

\newcommand{\bk}{{\mathbf{k}}}

\nc{\gG}{\mathcal{G}}
\nc{\gH}{\mathcal{H}}
\nc{\hK}{H}
\nc{\hH}{K}

\nc{\difc}{difference\xspace}
\nc{\difcop}{difference operator\xspace}
\nc{\difcops}{difference operators\xspace}
\nc{\Difcop}{Difference operator\xspace}
\nc{\Difcops}{Difference operators\xspace}
\nc{\difchopf}{difference Hopf algebra\xspace}
\nc{\difchopfs}{difference Hopf algebras\xspace}
\nc{\Difchopf}{Difference Hopf algebra\xspace}
\nc{\Difchopfs}{Difference Hopf algebras\xspace}

\nc{\vep}{\varepsilon}
\nc{\bin}[2]{ (_{\stackrel{\scs{#1}}{\scs{#2}}})}  
\nc{\binc}[2]{(\!\! \begin{array}{c} \scs{#1}\\
		\scs{#2} \end{array}\!\!)}  
\nc{\bincc}[2]{  ( {\scs{#1} \atop
		\vspace{-1cm}\scs{#2}} )}  
\nc{\oline}[1]{\overline{#1}}
\nc{\mapm}[1]{\lfloor\!|{#1}|\!\rfloor}
\nc{\bs}{\bar{S}}
\nc{\la}{\longrightarrow}
\nc{\ot}{\otimes}
\nc{\rar}{\rightarrow}
\nc{\con}{\ast}
\nc{\dar}{\downarrow}
\nc{\dap}[1]{\downarrow \rlap{$\scriptstyle{#1}$}}
\nc{\defeq}{\stackrel{\rm def}{=}}
\nc{\dis}[1]{\displaystyle{#1}}
\nc{\dotcup}{\ \displaystyle{\bigcup^\bullet}\ }
\nc{\hcm}{\ \hat{,}\ }
\nc{\hts}{\hat{\otimes}}
\nc{\hcirc}{\hat{\circ}}
\nc{\lleft}{[}
\nc{\lright}{]}
\nc{\curlyl}{\left \{ \begin{array}{c} {} \\ {} \end{array}
	\right .  \!\!\!\!\!\!\!}
\nc{\curlyr}{ \!\!\!\!\!\!\!
	\left . \begin{array}{c} {} \\ {} \end{array}
	\right \} }
\nc{\longmid}{\left | \begin{array}{c} {} \\ {} \end{array}
	\right . \!\!\!\!\!\!\!}
\nc{\ora}[1]{\stackrel{#1}{\rar}}
\nc{\ola}[1]{\stackrel{#1}{\la}}
\nc{\scs}[1]{\scriptstyle{#1}} \nc{\mrm}[1]{{\rm #1}}
\nc{\dirlim}{\displaystyle{\lim_{\longrightarrow}}\,}
\nc{\invlim}{\displaystyle{\lim_{\longleftarrow}}\,}
\nc{\dislim}[1]{\displaystyle{\lim_{#1}}} \nc{\colim}{\mrm{colim}}
\nc{\mvp}{\vspace{0.3cm}} \nc{\tk}{^{(k)}} \nc{\tp}{^\prime}
\nc{\ttp}{^{\prime\prime}} \nc{\svp}{\vspace{2cm}}
\nc{\vp}{\vspace{8cm}}
\nc{\modg}[1]{\!<\!\!{#1}\!\!>}
\nc{\intg}[1]{F_C(#1)}
\nc{\lmodg}{\!<\!\!}
\nc{\rmodg}{\!\!>\!}
\nc{\cpi}{\widehat{\Pi}}
\nc{\ssha}{{\mbox{\cyrs X}}} 
\nc{\tsha}{{\mbox{\cyrt X}}}
\nc{\shpr}{\diamond}    
\nc{\labs}{\mid\!}
\nc{\rabs}{\!\mid}

\newcommand{\Ad}{\mathrm{Ad}}
\nc{\ad}{\mrm{ad}}
\nc{\ann}{\mrm{ann}}
\nc{\Aut}{\mrm{Aut}}
\nc{\Der}{\mrm{Der}}
\nc{\br}{\mrm{bre}}
\nc{\can}{\mrm{can}}
\nc{\Cont}{\mrm{Cont}}
\nc{\rchar}{\mrm{char}}
\nc{\cok}{\mrm{coker}}
\nc{\de}{\mrm{dep}}
\nc{\dtf}{{R-{\rm tf}}}
\nc{\dtor}{{R-{\rm tor}}}

\nc{\Dif}{\mrm{Diff}}
\nc{\Div}{\mrm{Div}}
\nc{\End}{\mrm{End}}
\nc{\Ext}{\mrm{Ext}}
\nc{\Fil}{\mrm{Fil}}
\nc{\Fr}{\mrm{Fr}}
\nc{\Frob}{\mrm{Frob}}
\nc{\Gal}{\mrm{Gal}}
\nc{\GL}{\mrm{GL}}
\nc{\Gr}{\mrm{Gr}}
\nc{\Hom}{\mrm{Hom}}
\nc{\Hop}{\mrm{Hopf}}
\nc{\Hoch}{\mrm{Hoch}}
\nc{\hsr}{\mrm{H}}
\nc{\hpol}{\mrm{HP}}
\nc{\id}{\mrm{id}}
\nc{\im}{\mrm{im}}
\nc{\inv}{\mrm{inv}}
\nc{\Id}{\mrm{Id}}
\nc{\ID}{\mrm{ID}}
\nc{\Irr}{\mrm{Irr}}
\nc{\incl}{\mrm{incl}}
\nc{\length}{\mrm{length}}
\nc{\NLSW}{\mrm{NLSW}}
\nc{\Lie}{\mrm{Lie}}
\nc{\mchar}{\rm char}
\nc{\mpart}{\mrm{part}}
\nc{\ql}{{\QQ_\ell}}
\nc{\qp}{{\QQ_p}}
\nc{\rank}{\mrm{rank}}
\nc{\rcot}{\mrm{cot}}
\nc{\rdef}{\mrm{def}}
\nc{\rdiv}{{\rm div}}
\nc{\rtf}{{\rm tf}}
\nc{\rtor}{{\rm tor}}
\nc{\res}{\mrm{res}}
\nc{\SL}{\mrm{SL}}
\nc{\Spec}{\mrm{Spec}}
\nc{\tor}{\mrm{tor}}
\nc{\Tr}{\mrm{Tr}}
\nc{\tr}{\mrm{tr}}
\nc{\wt}{\mrm{wt}}

\newcommand{\co}{\mathsf{cosh}}

\nc{\bfk}{{\bf k}}
\nc{\bfone}{{\bf 1}}
\nc{\bfzero}{{\bf 0}}
\nc{\detail}{\marginpar{\bf More detail}
	\noindent{\bf Need more detail!}
	\svp}
\nc{\gap}{\marginpar{\bf Incomplete}\noindent{\bf Incomplete!!}
	\svp}
\nc{\FMod}{\mathbf{FMod}}
\nc{\Int}{\mathbf{Int}}
\nc{\Mon}{\mathbf{Mon}}
\nc{\remarks}{\noindent{\bf Remarks: }}
\nc{\Rep}{\mathbf{Rep}}
\nc{\Rings}{\mathbf{Rings}}
\nc{\Sets}{\mathbf{Sets}}

\nc{\BA}{{\mathbb A}}   \nc{\CC}{{\mathbb C}}
\nc{\DD}{{\mathbb D}}   \nc{\EE}{{\mathbb E}}
\nc{\FF}{{\mathbb F}}   \nc{\GG}{{\mathbb G}}
\nc{\HH}{{\mathbb H}}   \nc{\LL}{{\mathbb L}}
\nc{\NN}{{\mathbb N}}   \nc{\PP}{{\mathbb P}}
\nc{\QQ}{{\mathbb Q}}   \nc{\RR}{{\mathbb R}}
\nc{\TT}{{\mathbb T}}   \nc{\VV}{{\mathbb V}}
\nc{\ZZ}{{\mathbb Z}}   \nc{\TP}{\widetilde{P}}

\nc{\cala}{{\mathcal A}}    \nc{\calc}{{\mathcal C}}
\nc{\cald}{\mathcal{D}}     \nc{\cale}{{\mathcal E}}
\nc{\calf}{{\mathcal F}}    \nc{\calg}{{\mathcal G}}
\nc{\calh}{{\mathcal H}}    \nc{\cali}{{\mathcal I}}
\nc{\call}{{\mathcal L}}    \nc{\calm}{{\mathcal M}}
\nc{\caln}{{\mathcal N}}    \nc{\calo}{{\mathcal O}}
\nc{\calp}{{\mathcal P}}    \nc{\calr}{{\mathcal R}}
\nc{\cals}{{\mathcal S}}    \nc{\calt}{{\Omega}}
\nc{\calv}{{\mathcal V}}    \nc{\calw}{{\mathcal W}}
\nc{\calx}{{\mathcal X}}

\nc{\fraka}{{\mathfrak a}}
\nc{\frakb}{\mathfrak{b}}
\nc{\frakg}{{\frak g}}
\nc{\frakl}{{\frak l}}
\nc{\fraks}{{\frak s}}
\nc{\frakB}{{\frak B}}
\nc{\frakm}{{\frak m}}
\nc{\frakM}{{\frak M}}
\nc{\frakp}{{\frak p}}
\nc{\frakW}{{\frak W}}
\nc{\frakX}{{\frak X}}
\nc{\frakS}{{\frak S}}
\nc{\frakA}{{\frak A}}
\nc{\x}{{\frak x}}

\nc{\ynr}[1]{\textcolor{orange}{\underline{Yunnan:}#1 }}

\nc{\lir}[1]{\textcolor{red}{\underline{Li:}#1 }}

\begin{document}

\title[Crossed homomorphisms and Cartier-Kostant-Milnor-Moore theorem]{Crossed homomorphisms and Cartier-Kostant-Milnor-Moore theorem for difference Hopf algebras}

\author{Li Guo}
\address{
Department of Mathematics and Computer Science,
         Rutgers University,
         Newark, NJ 07102}
\email{liguo@rutgers.edu}

\author{Yunnan Li}
\address{School of Mathematics and Information Science, Guangzhou University,
Guangzhou 510006, China}
\email{ynli@gzhu.edu.cn}

\author{Yunhe Sheng}
\address{Department of Mathematics, Jilin University, Changchun 130012, Jilin, China}
\email{shengyh@jlu.edu.cn}

\author{Rong Tang}
\address{Department of Mathematics, Jilin University, Changchun 130012, Jilin, China}
\email{tangrong@jlu.edu.cn}
\date{\today}

\begin{abstract}
The celebrated Milnor-Moore theorem and the more general Cartier-Kostant-Milnor-Moore theorem establish close relationships of a connected and a pointed cocommutative Hopf algebra with its Lie algebra of primitive elements and its group of group-like elements. Crossed homomorphisms for Lie algebras, groups and Hopf algebras have been studied extensively, first from a cohomological perspective and then more broadly, with an important case given by \difc operators.
This paper shows that the relationship among the different algebraic structures captured in the Milnor-Moore theorem can be strengthened to include crossed homomorphisms and \difcops. We give a graph characterization of Hopf algebra crossed homomorphisms which are also compatible with the Milnor-Moore relation. We further investigate derived actions from crossed homomorphisms on groups, Lie algebras and Hopf algebras, and establish their relationship.
Finally we obtain a Cartier-Kostant-Milnor-Moore type structure theorem for pointed cocommutative difference Hopf algebras. Examples and classifications of difference operators are also provided for several Hopf algebras.
\end{abstract}

\subjclass[2010]{
16T05,	
17B40,	
22E60,	
16S30,	
17B35,	
16S40,	
39A70,	
16W99	
}

\keywords{crossed homomorphism; difference operator; Hopf algebra; Lie algebra; Lie group; universal enveloping algebra;  Milnor-Moore theorem; Cartier-Kostant-Milnor-Moore theorem}

\maketitle

\vspace{-1.2cm}
\tableofcontents

\vspace{-1.2cm}

\allowdisplaybreaks

\section{Introduction}
This paper studies the compatibility of crossed homomorphisms and in particular difference operators on  Lie algebras, groups and Hopf algebras, in the contexts of the Milnor-Moore theorem and the Cartier-Kostant-Milnor-Moore theorem.

Lie algebras, (Lie) groups and Hopf algebras  are fundamental algebraic structures. Their studies have greatly benefited from their close interconnections. Most notably,  the celebrated Milnor-Moore theorem  \mcite{MM} provides a Hopf algebra isomorphism from a connected cocommutative Hopf algebra over a field of characteristic zero
to the universal enveloping algebra of its Lie algebra of primitive elements, and the Cartier-Kostant-Milnor-Moore theorem~\mcite{Ca,Di} determines a pointed cocommutative Hopf algebra over a field of characteristic zero by its Lie algebra of primitive elements and its group of group-like elements.
Various generalizations of the theorems have been found over the years, including those for certain non-cocommutative Hopf algebras, generalized bialgebras, brace algebras, braided bialgebras and Hopf algebroids~\mcite{Ard,Cha,KM,Lod,LR,MM0,Ron}.
Recently the higher homotopy algebra version of the Milnor-Moore theorem is also obtained in \mcite{Fe}.

A common structure on these various algebraic structures is the crossed homomorphisms. Crossed homomorphisms on groups already appeared in Whitehead's work ~\mcite{Whi} in 1950 and were later applied to study non-abelian Galois cohomology~\mcite{Se} and
Banach modules of locally compact groups~\mcite{Da}.
The post-Lie Magnus expansion can be interpreted as a crossed homomorphism on (local) Lie groups~\mcite{MQ}.
The concept of crossed homomorphisms on Lie algebras was introduced in \mcite{Lue} in the study of non-abelian extensions of Lie algebras.
Recently crossed homomorphisms on Lie algebras  were used to construct actions of monoidal categories in the study of representations of Cartan type Lie algebras \mcite{PSTZ}. By differentiation, crossed homomorphisms on Lie groups give rise to crossed homomorphisms on the corresponding Lie algebras.

Crossed homomorphisms on Hopf algebras~\mcite{Sw} can be interpreted as 1-cocycles of the Sweedler cohomology of cocommutative Hopf algebras with coefficients in commutative module algebras. In the study of non-abelian Hopf cohomology~\mcite{NW1,NW2,NW3}, $1$-descent cocycles of a Hopf algebra $H$ with coefficients in a relative Hopf module $M$ serve the purpose of crossed homomorphisms on Hopf algebras. Among the recent studies, bijective Hopf algebra crossed homomorphisms were applied to construct solutions of the quantum Yang-Baxter equation~\mcite{AGV}, originated from the works of Etingof-Schedler-Soloviev and Lu-Yan-Zhu on the set-theoretical solutions of the quantum Yang-Baxter equation~\mcite{ESS,LYZ}.
Bijective crossed homomorphisms were also used to study Hopf-Galois structures on Galois field extensions~\mcite{Ts}.

Crossed homomorphisms with respect to the adjoint representation are in fact difference operators (also called differential operators of weight 1), abstracted from original instance in numerical analysis to algebraic settings of associative and Lie algebras~\mcite{GK,LGG,TFS}.
Differing by an affine transformation, another notion of difference algebra is also extensively studied in close connection with differential algebras~\mcite{Kol,Lev,PS,PS1}.
Recently, difference operators on groups were studied in ~\mcite{GLS}
as the inverse of Rota-Baxter operators on groups introduced there with motivation from integrable systems~\mcite{RS1,STS}.
The differentiation of a \difcop on a Lie group gives a \difcop on the corresponding Lie algebra.

The extensive study of crossed homomorphisms and difference operators on groups, Lie algebras and Hopf algebras naturally led us to investigate their interactions across the different algebraic structures.
In doing so, we obtain a Milnor-Moore theorem for crossed homomorphisms on connected cocommutative Hopf algebras. More precisely, we show that a crossed homomorphism  on a connected cocommutative Hopf algebra with respect to a module bialgebra structure on another connected cocommutative Hopf algebra is isomorphic to the universal enveloping of the induced  crossed homomorphism on the Lie algebra of primitive elements. Furthermore, we show that crossed homomorphisms can be characterized by their graphs and derived structures. Considering these properties on different algebraic structures give rise to other refinements of the Milnor-Moore theorem.
We also give a Cartier-Kostant-Milnor-Moore theorem when   \difcops are applied. More generally, we introduce the notion of a \difc  module bialgebra over a \difc Hopf algebra, and construct the corresponding smash product Hopf algebra.

Overall, the organization of the paper is as follows.

In Section~\mref{sec:mmtch}, we recall the notions of crossed homomorphisms and difference operators on Lie algebras, groups and Hopf algebras. We then show that crossed  homomorphisms  are compatible with the Milnor-Moore theorem for connected cocommutative Hopf algebras (Theorem~\mref{CMM-Diff-hopf}), leading to a structure theorem of connected cocommutative difference Hopf algebras (Corollary~\mref{coro:mm-difc}).

In Section~\mref{sec:graph}, we characterize crossed homomorphisms on Lie algebras, groups and Hopf algebras in terms of their graphs. We then show that taking the universal enveloping algebra of the graph of a Lie algebra crossed homomorphism give the graph of corresponding Hopf algebra crossed homomorphism (Theorem~\mref{prop:Liegr-leftgr}).

It is known that crossed homomorphisms on Lie algebras induce new Lie algebra actions. In Section~\mref{sec:derived}, we show that this can also be done for crossed homomorphisms on (Lie) groups and Hopf algebras. Furthermore, the differentiation from a Lie group to its Lie algebra induces the differentiation from the derived actions on the Lie group to the derived actions on the corresponding Lie algebra (Theorem~\mref{thm:diffcrossed}). In addition, there is a derived action of a Hopf algebra which restricts to derived actions of the Lie algebra of primitive elements and of the group of group-like elements (Theorem~\mref{th:new-action} and Corollary~\mref{co:derived}).

In  Section~\mref{sec:pcdha}, we first extend the notion of an $H$-module bialgebra from when $H$ is a bialgebra to when $H$ is a difference bialgebra. We then utilize this notion to obtain smash products of difference Hopf algebras. We prove the Cartier-Kostant-Milnor-Moore theorem  for pointed cocommutative difference Hopf algebras (Theorem~\mref{thm:dckmmt}).

Finally in Section~\mref{sec:exam}, we give examples and classifications of difference operators on some Hopf algebras, including the tensor Hopf algebra, Sweedler's 4-dimensional Hopf algebra $H_4$ and the Kac-Paljutkin Hopf algebra $H_8$.


\noindent
{\bf Conventions.}
In this paper, we fix a ground field $\bk$ of characteristic $0$. All the objects
under discussion, including vector spaces, algebras and tensor products, are taken over $\bk$ unless otherwise specified.
We will use roman letters such as $G, H$ for associative algebras and Hopf algebras, Fraktur letters such as $\mathfrak{g}, \mathfrak{h}$ for Lie algebras, and calligraphic letters such as $\mathcal{G}, \mathcal{H}$ for groups.
For a coalgebra $(C,\Delta,\vep)$, we abbreviate the Sweedler notation of the comultiplication $\Delta$ to
$$\Delta(x)=x_1\otimes x_2.$$
More generally, for $n\geq1$ we write $$\Delta^{(n)}(x)=(\Delta\otimes\id^{\otimes (n-1)})\cdots(\Delta\otimes\id)\Delta(x)=x_1\otimes\cdots\otimes x_{n+1}.$$
We follow~\mcite{Mon,Ra} for other basic notions of Hopf algebras.

\section{Milnor-Moore theorem for Hopf algebra crossed homomorphisms and \difchopfs}
\mlabel{sec:mmtch}
The celebrated Milnor-Moore theorem establishes a close relationship between a connected cocommutative Hopf algebra and its Lie algebra of primitive elements. In this section, we extend this relationship to crossed homomorphisms on connected cocommutative Hopf algebras and Lie algebras. As a special case, we obtain an enriched Milnor-Moore theorem for connected cocommutative \difchopfs.

\subsection{Crossed homomorphisms and \difcops}
\mlabel{ss:ch}
The notion of crossed homomorphisms, often called $1$-cocycles, have been defined in the context of non-abelian cohomology for various algebraic structures, including Lie algebras, groups and Hopf algebras~\mcite{Lue,Se,Sw}. An important special case is the difference operators in the corresponding categories~\mcite{GK,GLS}.

\begin{defn}
	\mlabel{crossed-homo-defi}
\begin{enumerate}
	\item
  Let $\phi:\g\to\Der(\h)$ be an action of a Lie algebra $(\g,[\cdot,\cdot]_\g)$  on a Lie algebra  $(\mathcal{\h},[\cdot,\cdot]_\h)$.  A linear map $d:\g\to \h$ is called a {\bf crossed homomorphism} on $\g$ with respect to the action $(\h,\phi)$ if
\begin{eqnarray}\mlabel{crossed-homo-algebra}
d[x,y]_\g=\phi(x)(d (y))-\phi(y)(d( x))+[d( x),d( y)]_\h,\quad \forall x,y\in \g.
\end{eqnarray}

\item
Let $\Phi:\gG\to \Aut(\gH)$ be an action of a group $\gG$ on a group $\gH$.  A set map $D:\gG\to \gH$ is called a {\bf crossed homomorphism} on $\gG$ with respect to the action $(\gH,\Phi)$ if
\begin{eqnarray}\mlabel{crossed-homo-group}
D(gh)=D(g)\Phi(g)(D(h)),\quad\forall\,g,h\in \gG.
\end{eqnarray}

\item
Let $\hK$ and $\hH$ be Hopf algebras such that $\hH$ is an $\hK$-module algebra in the sense that there is an algebra action $\rightharpoonup:\hK\to \End(\hH)$ such that
$$a\rightharpoonup 1=\vep(a)1,\quad a\rightharpoonup (xy)=(a_1\rightharpoonup x)(a_2\rightharpoonup y), \quad \forall a\in \hK, x, y \in \hH.$$
A coalgebra homomorphism $\pi:\hK\to \hH$ is called a {\bf crossed homomorphism} on $\hK$ with respect to the action $(\hH,\rightharpoonup)$ if
\begin{equation}\mlabel{eq:crosshomo-Hopf}
\pi(ab)=\pi(a_1)(a_2\rightharpoonup \pi(b)),\quad \forall\,a,b\in \hK.
\end{equation}
\end{enumerate}
\end{defn}

A {\bf homomorphism} between crossed homomorphisms on Lie algebras, groups or Hopf algebras is defined as expected. For Hopf algebras, such a homomorphism from $\pi:\hK\to \hH$ to  $\pi':\hK'\to \hH'$ consists of a pair of Hopf algebra homomorphisms $f:\hH\to \hH'$ and $g:\hK\to \hK'$ such that
$$f\pi=\pi'g,\quad f(a\rightharpoonup x)=g(a)\rightharpoonup f(x) ,\quad \forall\,x\in \hH,\,a\in \hK.$$

When a crossed homomorphism on either a Lie algebra, a group or a Hopf algebra is defined with respect to the adjoint action, then the crossed homomorphism is called a {\bf difference operator} on the corresponding structure~\mcite{GK,GLS,LGG}.
More precisely,
\begin{enumerate}
	\item
a {\bf \difc operator on a Lie algebra $\frakg$} is a linear operator $d$ on $\frakg$ such that
$$ d([x,y])_\frakg = [d(x),y]_\frakg + [x,d(y)]_\frakg +[d(x),d(y)]_\frakg, \quad \forall x, y\in \frakg;$$
\item a {\bf \difc operator on a group $\gG$} is a map $D:\gG\to \gG$ such that
$$ D(gh)=D(g)\Ad_g D(h), \quad \forall g, h\in \gG;$$
\item a {\bf \difc operator on a Hopf algebra $H$} is a coalgebra homomorphism $D:H\to H$ such that
	\begin{equation}\mlabel{eq:diff-Hopf}
	D(xy)=D(x_1)\ad_{x_2}D(y)=D(x_1)x_2D(y)S(x_3),\quad \forall\,x,y\in H.
\end{equation}
Here the (left) {\bf adjoint action} of $H$ on itself is defined by
\[\ad_x y:=x_1yS(x_2),\,\quad \forall x,y\in H.\]
\end{enumerate}

The term difference operator comes from its origin of the difference operator on functions, defined by
$$ D(f)(x):=f(x+1)-f(x).$$
As shown in~\mcite{GK}, the operator satisfies the operator identity
$$ D(fg)=D(f)g+fD(g)+D(f)D(g),$$
as a special case of the differential operator of weight $\lambda$ when $\lambda=1$. Another meaning of the term difference operator, commonly used in differential algebra and difference algebra~\mcite{Kol,Lev}, is that the operator is an injective algebra homomorphism.

\begin{remark}
\begin{enumerate}
	\item A crossed homomorphism on either a Lie algebra, a group or a Hopf algebra with respect to the trivial action is simply a homomorphism of the corresponding structure. 
\item Bijective Hopf algebra crossed homomorphisms
are the same as
bijective 1-cocycles defined in \cite[Definition~1.10]{AGV} to construct Hopf braces.
\end{enumerate}
\end{remark}

We will show below that the notions of crossed homomorphisms and \difc operators on Hopf algebras and Lie algebras are compatible in the context of the Milnor-Moore theorem.

\begin{lemma}\mlabel{le:crosshomo-anti}
For a crossed homomorphism $\pi:\hK\to \hH$, we have
\begin{enumerate}
\item
$\pi(1)=1$;\mlabel{it:crossp1}
\item
$S_\hH\pi(a)=a_1\rightharpoonup \pi S_\hK(a_2),\quad \pi S_\hK(a)=S_\hK(a_1)\rightharpoonup  S_\hH\pi(a_2),\quad\forall a\in \hK$;\mlabel{it:crossp2}
\item
the convolution inverse of $\pi$ in $\Hom_\bk(\hK,\hH)$ is $S_\hH\pi$.\mlabel{it:crossp3}
\end{enumerate}
\end{lemma}
\begin{proof}
\meqref{it:crossp1} Since $\pi$ is a coalgebra homomorphism, $\pi(1)$ has inverse $S_\hH\pi(1)$. By Eq.~\meqref{eq:crosshomo-Hopf}, we have $\pi(1)=\pi(1)^2$.  Thus $\pi(1)=1$.

\smallskip
\meqref{it:crossp2} Applying $\pi(1)=1$ and Eq.~\meqref{eq:crosshomo-Hopf}, we first have
$$\vep_\hK(a)1=\pi(a_1S_\hK(a_2))=\pi(a_1)(a_2\rightharpoonup  \pi S_\hK(a_3)), \quad \forall a\in \hK.$$
As $\pi$ is a coalgebra homomorphism, it implies that
$$S_\hH\pi(a)=S_\hH\pi(a_1)\vep_\hK(a_2)=S_\hH\pi(a_1)\pi(a_2)(a_3\rightharpoonup  \pi S_\hK(a_4))= a_1\rightharpoonup  \pi S_\hK(a_2).$$
In the same way,
$$\pi S_\hK(a)=S_\hK(a_1)a_2\rightharpoonup  \pi S_\hK(a_3)=
S_\hK(a_1)\rightharpoonup  (a_2\rightharpoonup  \pi S_\hK(a_3))
=S_\hK(a_1)\rightharpoonup  S_\hH\pi(a_2).$$

\smallskip
\meqref{it:crossp3} Since $\pi$ is a coalgebra homomorphism, we have
\begin{align*}
&\pi(a_1)S_\hH\pi(a_2)=\pi(a)_1S_\hH(\pi(a)_2)=\vep_\hH(\pi(a))1=\vep_\hK(a)1,\\
&S_\hH\pi(a_1)\pi(a_2)=S_\hH(\pi(a)_1)\pi(a)_2=\vep_\hH(\pi(a))1=\vep_\hK(a)1, \quad \forall a\in \hK.
\end{align*}
This is what we need.
\end{proof}

\subsection{Milnor-Moore theorem and crossed homomorphisms}
\mlabel{ss:mmt}
We next establish the relationship between the crossed homomorphisms on Lie algebras and Hopf algebras in the context of the Milnor-Moore theorem.

Let $C$ be a  Hopf algebra and $c\in C$.
\begin{itemize}
  \item[{\rm(i)}] $c$ is called a {\bf group-like} element, if $\Delta(c)=c\otimes c$ and $\vep(c)=1$.
Denote by $G(C)$ the group of group-like elements in $C$.

 \item[{\rm(ii)}] If $\Delta(c)=c\otimes 1+1\otimes c$, then $c$ is called a {\bf primitive element}.
 Let $P(C)$ denote the set of primitive elements in $C$, which is a Lie algebra.
\end{itemize}

\begin{prop}\mlabel{prop:restrict-cross}
Let $\pi:\hK\to \hH$ be a crossed homomorphism on a  Hopf algebra $\hK$ with respect to a module bialgebra action $\rightharpoonup:\hK \to \End(\hH)$,  namely, the $\hK$-module algebra action $\rightharpoonup$ is also compatible with the coalgebra structure of $\hH$ as follows,
$$\vep_\hH(a\rightharpoonup x)=\vep_\hK(a)\vep_\hH(x),\quad \Delta_\hH(a\rightharpoonup x)=(a_1\rightharpoonup x_1)\otimes(a_2\rightharpoonup x_2), \quad \forall a\in \hK, x \in \hH.$$
Then
\begin{enumerate}
\item $\pi|_{G(\hK)}$ is a group crossed homomorphism from  $G(\hK)$ to $G(\hH)$;
\mlabel{it:rest1}
\item $\pi|_{P(\hK)}$ is a Lie algebra crossed homomorphism from $P(\hK)$ to $P(\hH)$.
\mlabel{it:rest2}
\end{enumerate}
\end{prop}
\begin{proof}
Since $\pi$ is a coalgebra homomorphism,
$\pi|_{G(\hK)}$ is a set map from $G(\hK)$ to $G(\hH)$, and $\pi|_{P(\hK)}$ is a linear map from $P(\hK)$ to $P(\hH)$. Also, we have the restricted actions $\rightharpoonup:G(\hK)\to \Aut(G(\hH))$ and $\rightharpoonup:P(\hK)\to \Der(P(\hH))$, as $\hH$ is an $\hK$-module bialgebra via $\rightharpoonup$.

To obtain \meqref{it:rest1}, just note that Eq.~\meqref{eq:crosshomo-Hopf} indeed restricts to Eq.~\meqref{crossed-homo-group} when $a,b\in G(\hK)$.

To obtain \meqref{it:rest2}, applying $\pi(1)=1$, we first see that
$$\pi(ab)=\pi(a)\pi(b)+a\rightharpoonup \pi(b),\quad \forall\,a\in P(\hK),b\in \hK.$$
Hence, for $a,b\in P(\hK)$,
\begin{align*}
\pi([a,b])&=\pi(ab-ba)\\
&=(\pi(a)\pi(b)+a\rightharpoonup \pi(b))-(\pi(b)\pi(a)+b\rightharpoonup \pi(a))\\
&=a\rightharpoonup \pi(b) - b\rightharpoonup \pi(a) + [\pi(a),\pi(b)],
\end{align*}
showing that $\pi|_{P(\hK)}$ satisfies Eq.~\meqref{crossed-homo-algebra}.
\end{proof}

\begin{remark}
When $\hH$ and $\hK$ are cocommutative and the crossed homomorphism $\pi:\hK\to \hH$ is bijective,
Proposition~\mref{prop:restrict-cross} \meqref{it:rest2} recovers \cite[Lemma 4.2]{AGV}. In fact, bijective 1-cocycles from $\hK$ to $\hH$ defined in \mcite{AGV} only require an $\hK$-module algebra action $\rightharpoonup$ on $\hH$. When $\hH$ and $\hK$ are cocommutative, such a module algebra action must be a module bialgebra action.
\end{remark}

Since the adjoint action $\ad$ of $H$ induces adjoint actions of $G(H)$ and $P(H)$ respectively, we obtain
\begin{coro}\mlabel{coro:restrict-diff}
Let $D$ be a \difc operator on a Hopf algebra $H$. Then
\begin{enumerate}
\item$D|_{G(H)}$ is a difference operator on the  group $G(H)$;
\item  $D|_{P(H)}$ is a difference operator on the  Lie algebra $ {P(H)}$.
\end{enumerate}
\end{coro}

Going in the opposite direction, one can extend a crossed homomorphism on a Lie algebra to its universal enveloping algebra. See \cite[Lemma 4.3, Lemma 4.4]{AGV}.
Let $\phi:\g\to\Der(\h)$ be a Lie algebra action of $(\g,[\cdot,\cdot]_\g)$ on  $(\mathcal{\h},[\cdot,\cdot]_\h)$. Then $\phi$ can be extended to a module bialgebra action  $\bar\phi:U(\g)\to \End(U(\h))$ defined by
$$
\bar\phi(x)(1)=0,\,\bar\phi(x)(y_1\cdots y_r)=\sum_{i=1}^r y_1\cdots y_{i-1} \phi(x)(y_i) y_{i+1} \cdots y_r,
\quad\forall x\in\g,\,y_1,\dots, y_r\in\h,\,r\geq1.$$

\begin{prop}\mlabel{prop:cross-uea}
A Lie algebra crossed homomorphism  $\pi:\g\to \h$ with respect to the action $(\h,\phi)$ can be extended to a unique Hopf algebra crossed homomorphism $\bar\pi:U(\g)\to U(\h)$ with respect to the extended module bialgebra action $\bar\phi:U(\g)\to \End(U(\h))$.
More precisely, $\bar{\pi}:U(\frak g)\longrightarrow U(\frak h)$ is given by
$$\bar{\pi}(x_1\cdots x_n)=(\pi(x_1)+\bar\phi(x_1))\cdots(\pi(x_n)+\bar\phi(x_n))(1),
\quad\forall x_1,\dots,x_n\in\frak g,\,n\geq1,$$
where $\pi(x_k), 1\leq k\leq n$ are left multiplications.
\end{prop}

Recall that, under our running hypothesis of characteristic zero for the base field, for a connected cocommutative Hopf algebra $H$, the Milnor-Moore theorem gives a Hopf algebra isomorphism $\varphi_H:H\to U(P(H))$ of $H$ with the universal enveloping algebra of the Lie algebra $P(H)$ of primitive elements of $H$. We now give the following compatibility of the Milnor-Moore theorem for crossed homomorphisms.
\begin{theorem} $(${\bf Milnor-Moore Theorem for Crossed Homomorphisms}$)$
Let $\hK$ and $\hH$ be connected cocommutative Hopf algebras  such that $K$ is an $H$-module bialgebra   via an action $\rightharpoonup:\hK\to \End(\hH)$. Then a Hopf algebra crossed homomorphism $\pi:\hK\to \hH$ with respect to $\rightharpoonup$ induces a unique Hopf algebra crossed homomorphism $\bar\pi:U(P(\hK))\to U(P(\hH))$ such that the following diagram of crossed homomorphisms and inclusions is commutative$:$
\begin{equation} \mlabel{eq:diag}
\begin{split}
\xymatrix{\hK \ar[rrr]^{\pi} &  &  &  \hH \\
&P(\hK)\ar[ld]^{j_\hK}\ar[lu]_{i_\hK}\ar[r]^{\pi|_{P(\hK)}}&P(\hH)\ar[rd]_{j_\hH}\ar[ru]^{i_\hH}&
\\
U(P(\hK))\ar[uu]^{\varphi_\hK} \ar@{.>}[rrr]^{\bar\pi} & & & U(P(\hH))\ar[uu]_{ \varphi_\hH}}
\end{split}
\end{equation}
Moreover, the Milnor-Moore isomorphisms $\varphi_\hK$ and $\varphi_\hH$ give an isomorphism between the Hopf algebra crossed homomorphisms $\pi:\hK\to \hH$ and $\bar\pi:U(P(\hK))\to U(P(\hH))$.
\mlabel{CMM-Diff-hopf}
\end{theorem}
\begin{proof}
First by Proposition~\mref{prop:restrict-cross}, we obtain a crossed homomorphism $\pi|_{P(\hK)}:P(\hK)\to P(\hH)$ of Lie algebras by restricting $\pi:\hK\to \hH$ to the Lie algebra of  primitive elements. Then by Proposition ~\mref{prop:cross-uea}, it can be extended to a unique Hopf algebra crossed homomorphism $\bar\pi:U(P(\hK))\to U(P(\hH))$. This proves the commutativity of the two trapezoids in the middle of the diagram in Eq.~\meqref{eq:diag}.

On the other hand, by the classical Milnor-Moore theorem,
the embedding $i_\hK:P(\hK)\to \hK$ induces a Hopf algebra isomorphism $\varphi_\hK: U(P(\hK))\to \hK$, and the same holds for $\hH$, giving the commutativity of the two triangles.

Before continuing, we verify the following compatibility condition:
\begin{equation} \varphi_\hH\left(a\rightharpoonup v\right)=\varphi_\hK(a)\rightharpoonup \varphi_\hH(v),\quad\forall v\in U(P(\hH)),\,a\in P(\hK).
\mlabel{eq:club}
\end{equation}
Indeed, since $\hH$ is an $\hK$-module bialgebra, and $U(P(\hH))$ is a $U(P(\hK))$-module bialgebra,
for all $v=x_1\cdots x_r$ with $x_i\in P(\hH), 1\leq i\leq r$, we derive
\begin{align*}
\varphi_\hH\left(a\rightharpoonup v\right)
&=\varphi_\hH\left(a\rightharpoonup x_1\cdots x_r\right)\\
&=\varphi_\hH\left(\sum_{i=1}^r x_1\cdots x_{i-1}(a\rightharpoonup x_i)x_{i+1}\cdots x_r\right)\\
&=\sum_{i=1}^r \varphi_\hH(x_1)\cdots \varphi_\hH(x_{i-1})\varphi_\hH(a\rightharpoonup x_i)\varphi_\hH(x_{i+1})\cdots \varphi_\hH(x_r)\\
&=\sum_{i=1}^r \varphi_\hH(x_1)\cdots \varphi_\hH(x_{i-1})(\varphi_\hK(a)\rightharpoonup \varphi_\hH(x_i))\varphi_\hH(x_{i+1})\cdots \varphi_\hH(x_r)\\
&=\varphi_\hK(a)\rightharpoonup\varphi_\hH(x_1)\cdots \varphi_\hH(x_r)\\
&=\varphi_\hK(a)\rightharpoonup \varphi_\hH(v),
\end{align*}
as needed.

Next we prove the commutativity of the outer rectangle in the diagram~\meqref{eq:diag}, namely the equality
\begin{equation} \varphi_\hH(\bar\pi(u))=\pi(\varphi_\hK(u)),\quad\forall u\in U(P(\hK)).
\mlabel{eq:ast1}
\end{equation}
We apply induction on the degree of $u$ via the coradical filtration $\{U(P(\hK))_n\}_{n\geq0}$ of $U(P(\hK))$.

First note that $\varphi_\hH \bar\pi j_\hK=\pi \varphi_\hK j_\hK$ from the above discussion. So Eq.~\meqref{eq:ast1} clearly holds when $u\in U(P(\hK))_1$. Now suppose that Eq.~\meqref{eq:ast1} holds for all $u\in U(P(\hK))_k$ with $k\geq1$. To complete the inductive step, we only need to check elements of the form $au$
with $a\in P(\hK)$, as they span $U(P(\hK))_{k+1}$. The left hand side of Eq.~\meqref{eq:ast1} for $au\in U(P(\hK))_{k+1}$ is
\begin{align*}
\varphi_\hH(\bar\pi(au))
&=\varphi_\hH\left(\bar\pi(a)\bar\pi(u)+ a\rightharpoonup \bar\pi(u)\right)\\
&=\varphi_\hH\left(\bar\pi(a)\bar\pi(u)\right)+\varphi_\hH\left(a\rightharpoonup \bar\pi(u)\right)\\
&=\varphi_\hH(\bar\pi(a))\varphi_\hH(\bar\pi(u))+\varphi_\hK(a)\rightharpoonup \varphi_\hH(\bar\pi(u))\\
&=\pi(\varphi_\hK(a))\pi(\varphi_\hK(u))+\varphi_\hK(a)\rightharpoonup \pi(\varphi_\hK(u))\\
&=\pi(\varphi_\hK(a)\varphi_\hK(u))\\
&=\pi(\varphi_\hK(au)),
\end{align*}
where the first and the second to the last equalities use Eq.~\meqref{eq:crosshomo-Hopf}, the third equality is due to Eq.~\meqref{eq:club} and the fourth one is obtained by the induction hypothesis.

In order to show that $\varphi_\hK$ and $\varphi_\hH$ give an isomorphism between the Hopf algebra crossed homomorphisms $\pi$ and $\bar\pi$, it remains to show
$$\varphi_\hH\left(u\rightharpoonup v\right)=\varphi_\hK(u)\rightharpoonup \varphi_\hH(v),\quad\forall v\in U(P(\hH)),\,u\in U(P(\hK)),$$
which we again prove by induction on the degree of $u$. When $u\in U(P(\hK))_1$, this is just due to Eq.~\meqref{eq:club}. Now suppose that it holds for all $u\in U(P(\hK))_k$ with $k\geq1$, then given any $a\in P(\hK)$,  we only need to check $au \in U(P(\hK))_{k+1}$ as noted in the previous paragraph. In this case we have
\begin{align*}
\varphi_\hH\left(au\rightharpoonup v\right)
&=\varphi_\hH\left(a\rightharpoonup (u\rightharpoonup v)\right)\\
&=\varphi_\hK(a)\rightharpoonup \varphi_\hH(u\rightharpoonup v)\\
&=\varphi_\hK(a)\rightharpoonup (\varphi_\hK(u)\rightharpoonup \varphi_\hH(v))\\
&=\varphi_\hK(a)\varphi_\hK(u)\rightharpoonup \varphi_\hH(v)\\
&=\varphi_\hK(au)\rightharpoonup \varphi_\hH(v),
\end{align*}
where the second and the third equalities are obtained by the induction hypothesis.
\end{proof}

In particular, when the crossed homomorphisms are \difcops, we obtain the following \difc enhancement of the Milnor-Moore isomorphism.
\begin{coro}\mlabel{coro:mm-difc}
Let  $D$ be a difference operator on a connected cocommutative Hopf algebra $H$. Then there is an isomorphism   from the difference operator $\bar{D}$ on $U(P(H))$ to the difference operator $D$ on $H$.
\end{coro}

At the end of this section, we establish a relation between difference cocommutative  Hopf algebras and Rota-Baxter   cocommutative  Hopf algebras.

In \mcite{Go}, Goncharov defined a {\bf Rota-Baxter operators
of weight 1} on a {\it cocommutative} Hopf algebra $H$ as a coalgebra homomorphism $B$ satisfying
\begin{equation}\mlabel{eq:rb-Hopf}
B(x)B(y)=B\left(x_1\ad_{B(x_2)}y\right)=B\left(x_1 B(x_2)y S(B(x_3))\right),\quad \forall\,x,y\in H.
\end{equation}

\begin{prop}
Let $H$ be a cocommutative Hopf algebra and let $D$ be a coalgebra automorphism on $D$. Then $D$ is a \difcop if and only if its inverse $D^{-1}$ is a Rota-Baxter operator of weight $1$.
\end{prop}
\begin{proof}
Write $B=D^{-1}$. Then $B$ is also a coalgebra homomorphism. If $D$ is a \difcop, then we use Eq.~\meqref{eq:diff-Hopf} to check that
\begin{align*}
B(x)B(y)&=B\left(D(B(x)B(y))\right)\\
&=B\left(D(B(x_1))B(x_2)D(B(y))S(B(x_3))\right)\\
&=B\left(x_1B(x_2)yS(B(x_3))\right)
\end{align*}
for all $x,y\in H$. Hence, $B$ satisfies Eq.~\meqref{eq:rb-Hopf} and is a Rota-Baxter operator of weight 1.

Conversely, if $B$ is a Rota-Baxter operator, then by Eq.~\meqref{eq:rb-Hopf},
\begin{align*}
D(xy)&=D\left(B(D(x))B(D(y))\right)\\
&=D\left(B\left(D(x_1)B(D(x_2))D(y)S(B(D(x_3)))\right)\right)\\
&=D(x_1)x_2D(y)S(x_3)
\end{align*}
for all $x,y\in H$. So $D$ is a \difcop.
\end{proof}

\section{Graph characterizations of crossed homomorphisms}
\mlabel{sec:graph}
In this section, we use the (left) graphs of maps to characterize crossed homomorphisms on Hopf algebras and study their relationship with the graphs of Lie algebra  crossed homomorphisms.

Let $\phi:\g\to\Der(\h)$ be an action of a Lie algebra  $(\g,[\cdot,\cdot]_\g)$  on   a Lie algebra   $(\mathcal{\h},[\cdot,\cdot]_\h)$.
Denote by $\frak h\rtimes\frak g$ the semi-direct product Lie algebra  of $\frak h$ and $\frak g$ with respect to the   action $(\h,\phi)$. More precisely, the Lie bracket $[\cdot,\cdot]_{\rtimes}:\wedge^2(\h\oplus\g)\to \h\oplus\g$ is given by
\begin{eqnarray*}
	~[(u,x),(v,y)]_{\rtimes}&=&([u,v]_\frak h+\phi(x)(v)-\phi(y)(u),[x,y]_\frak g),\quad \forall x,y\in \g, ~u,v\in \h.
\end{eqnarray*}

It is proved in \mcite{PSTZ} that a linear map $\pi:\g\to \h$ is a Lie algebra crossed homomorphism with respect to the action $(\h,\phi)$  if and only if the graph of $\pi$, $$\Gr_\pi:=\{(\pi(x),x)\,|\,x\in\frak g\}$$
is a Lie subalgebra of $\frak h\rtimes\frak g$. Similar graph characterizations can also be found for other algebraic structures. See~\mcite{Uch} for example.

For Hopf algebras $\hK$ and $\hH$, let $\rightharpoonup:\hK\to \End(\hH)$ be an action such that $\hH$ is an $\hK$-module algebra.  The {\bf smash product} of Hopf algebras $\hH$ and $\hK$ is defined to be $\hH\otimes \hK$ with the multiplication
$$(x\# a)(y\# b)=x(a_1\rightharpoonup y)\# a_2b, \quad \forall x,y\in \hH,\, a,b\in \hK,$$
with $x\otimes a\in \hH\otimes \hK$ denoted by $x\# a$. We denote such a smash product algebra by $\hH\# \hK$.
If $\hK$ is cocommutative  and $\hH$ is further an $\hK$-module bialgebra,
then $\hH\# \hK$ becomes a Hopf algebra with the usual tensor product comultiplication and the antipode defined by
$$S(x\# a):=(S_\hK(a_1)\rightharpoonup S_\hH(x))\#S_\hK(a_2).$$

\begin{defn}
Given any coalgebra homomorphism $\pi:\hK\to \hH$, the {\bf graph} of $\pi$, denoted by $\Gr_{\pi}$, is defined to be the subspace $\im((\pi\otimes\id)\Delta_\hK)$ of $\hH\otimes \hK$, that is,
$$\Gr_{\pi}\coloneqq \{\pi(a_1)\otimes a_2\,|\,a\in \hK\}.$$
\end{defn}

\begin{theorem}
Let $\hK$ and $\hH$ be Hopf algebras with an algebra action  $\rightharpoonup:\hK\to \End(\hH)$ such that $\hH$ is an $\hK$-module bialgebra.
A coalgebra homomorphism $\pi:\hK\to \hH$ is a Hopf algebra crossed homomorphism with respect to the action $\rightharpoonup$ if and only if its graph $\Gr_{\pi}$ is a subalgebra of the smash product algebra $\hH\#\hK$.
\mlabel{th:leftgr}
\end{theorem}
\begin{proof}
If $\pi:\hK\to \hH$ is a crossed homomorphism, then
$$(\pi(a_1)\# a_2)(\pi(b_1)\# b_2) = \pi(a_1)(a_2\rightharpoonup \pi(b_1))\# a_3b_2 = \pi(a_1b_1)\#a_2b_2\in \Gr_{\pi},\quad \forall\,a,b\in \hK.$$
Also, $1\#1=\pi(1)\#1\in \Gr_{\pi}$. Hence, $\Gr_{\pi}$ is a subalgebra of $\hH\#\hK$.

Conversely, assume that $\Gr_{\pi}$ is a subalgebra of $\hH\#\hK$. Then for any $a,b\in \hK$, there exists $w\in \hK$ such that
$$\pi(w_1)\# w_2=(\pi(a_1)\# a_2)(\pi(b_1)\# b_2) = \pi(a_1)(a_2\rightharpoonup \pi(b_1))\# a_3b_2.$$
In particular,
$$w=\vep_\hH(\pi(w_1))w_2=\vep_\hH(\pi(a_1)(a_2\rightharpoonup \pi(b_1)))a_3b_2=\vep_\hH(a_1\rightharpoonup \pi(b_1))a_2b_2=ab,$$
as the counit $\vep_\hH$ of $\hH$ is compatible with $\rightharpoonup$. On the other hand,
$$\pi(w)=\pi(w_1)\vep_\hK(w_2)= \pi(a_1)(a_2\rightharpoonup \pi(b_1))\vep_\hK(a_3b_2) = \pi(a_1)(a_2\rightharpoonup \pi(b)),$$
so $\pi(ab)=\pi(a_1)(a_2\rightharpoonup \pi(b))$, showing that $\pi:\hK\to \hH$ is a crossed homomorphism.
\end{proof}

\begin{coro}\mlabel{coro:leftgr}
Let $\hH$ and $\hK$ be as in Theorem~\mref{th:leftgr} and with $\hK$ cocommutative. A coalgebra homomorphism $\pi:\hK\to \hH$ is a Hopf algebra crossed homomorphism with respect to the action $\rightharpoonup$ if and only if $\Gr_{\pi}$ is a Hopf subalgebra of the smash product Hopf algebra $\hH\# \hK$.

Moreover, a Hopf algebra crossed homomorphism induces a Hopf algebra isomorphism,
\begin{equation}\mlabel{eq:isom-leftgr}
\Psi:\hK\to \Gr_{\pi},\,a\mapsto \pi(a_1)\# a_2.
\end{equation}
\end{coro}
\begin{proof}
By Theorem~\mref{th:leftgr}, we only need to show that $\Gr_{\pi}$ is a subcoalgebra
of $\hH\# \hK$, and $S(\Gr_{\pi})\subseteq \Gr_{\pi}$. These statements follow from the cocommutativity of $\hK$:
$$\Delta(\pi(a_1)\# a_2)=(\pi(a_1)\# a_3)\otimes(\pi(a_2)\# a_4)=(\pi(a_1)\# a_2)\otimes(\pi(a_3)\# a_4)\in \Gr_{\pi}\otimes \Gr_{\pi},$$
and Lemma~\mref{le:crosshomo-anti}~\meqref{it:crossp2}:
$$S(\pi(a_1)\# a_2)=(S_\hK(a_1)\rightharpoonup S_\hH\pi(a_2))\# S_\hK(a_3)=\pi S_\hK(a_1)\# S_\hK(a_2)\in \Gr_{\pi}.$$

From the above discussion, one can easily see that $\Psi$ as stated is a Hopf algebra isomorphism with its inverse $\Psi^{-1}=\vep_\hH\otimes\id$.
\end{proof}

We now show that the universal  enveloping algebra
$U(\Gr_\pi)$ of the graph of a Lie algebra crossed homomorphism $\pi$ is given by the graph of the Hopf algebra crossed homomorphism $\bar{\pi}$ lifted from $\pi$.

\begin{theorem}\mlabel{prop:Liegr-leftgr}
Let $\pi:\g\to \h$ be a crossed homomorphism on a Lie algebra $\g$ with respect to an action $(\h,\phi)$ and let $\bar{\pi}:U(\g)\to U(\h)$ be the lifted Hopf algebra crossed homomorphism given in Proposition \mref{prop:cross-uea}.
 \begin{enumerate}
  \item \mlabel{it:grarel1} There is a Hopf algebra isomorphism $U(\frak h)\#U(\frak g)\simeq U(\frak h\rtimes\frak g)$, that is, the universal enveloping algebra of a Lie algebra semidirect product is isomorphic to the corresponding Hopf algebra smash product;

   \item \mlabel{it:grarel2} There is a Hopf algebra isomorphism $\Gr_{\bar{\pi}}\simeq U(\Gr_\pi),$ that is, the universal enveloping algebra of the graph of a Lie algebra crossed homomorphism is isomorphic to the graph of the lifted Hopf algebra crossed homomorphism.
 \end{enumerate}
\end{theorem}
\begin{proof}
\meqref{it:grarel1} Define a linear map $\varphi:\frak h\rtimes\frak g\to U(\frak h)\#U(\frak g)$ by
$$\varphi(u,x)= u\#1+1\#x,\quad\forall x\in \g, ~u\in \h.$$
First we check that $\varphi$ is a Lie algebra homomorphism, which follows from
\begin{align*}
\varphi([(u,x),(v,y)]_{\rtimes})&=\varphi([u,v]_\frak h+\phi(x)(v)-\phi(y)(u),[x,y]_\frak g)\\
&=([u,v]_\frak h+\phi(x)(v)-\phi(y)(u))\#1+1\#[x,y]_\frak g,\\
[\varphi(u,x),\varphi(v,y)]&=[u\#1+1\#x,v\#1+1\#y]\\
&=uv\#1+u\#y+\phi(x)(v)\#1+v\#x+1\#xy
\\&\quad-(vu\#1+v\#x+\phi(y)(u)\#1+u\#y+1\#yx)\\
&=([u,v]_\frak h+\phi(x)(v)-\phi(y)(u))\#1+1\#[x,y]_\frak g,
\end{align*}
for all $ x,y\in \g, ~u,v\in \h.$

By the universal property of $U(\frak h\rtimes\frak g)$,
there exists a unique algebra homomorphism
$$\bar\varphi:U(\frak h\rtimes\frak g)\to U(\frak h)\#U(\frak g)$$
such that $\varphi=\bar\varphi\, i_{\frak h\rtimes\frak g}$, where $i_{\frak h\rtimes\frak g}:\frak h\rtimes\frak g\to U(\frak h\rtimes\frak g)$ is the natural embedding.

Since the image ${\rm Im}\,\varphi=\{u\#1,1\#x\,|\,x\in\frak g, u\in\h\}$ generates $U(\frak h)\#U(\frak g)$ as an algebra, $\bar\varphi$ is surjective.
On the other hand, one checks that
$$\Delta\bar\varphi=(\bar\varphi\otimes\bar\varphi)\Delta$$
holds on ${\rm Im}\,\varphi$, thus it holds on $U(\frak h\rtimes\frak g)$ as an equality of algebra homomorphisms. It is obvious that $\vep\bar\varphi=\vep$. Therefore, $\bar\varphi$ is a Hopf algebra surjection. By the Heyneman-Radford theorem  (see 
\cite[Theorem 5.3.1]{Mon}), $\bar\varphi$ is also injective, as $U(\frak h\rtimes\frak g)_1=\bk\oplus(\frak h\rtimes\frak g)$ and $\bar\varphi|_{U(\frak h\rtimes\frak g)_1}$ is clearly injective by the definition of $\varphi$. In conclusion, $\bar\varphi$ is the desired Hopf algebra isomorphism.

\meqref{it:grarel2} By Corollary~\mref{coro:leftgr}, $\Gr_{\bar{\pi}}$ is a Hopf subalgebra of $U(\frak h)\#U(\frak g)$ isomorphic to $U(\frak g)$. Define $\psi\coloneqq \varphi|_{\Gr_\pi}$.  Then $\psi$ is injective and
$$\psi(\pi(x),x)=\pi(x)\#1+1\#x
=(\bar\pi\otimes\id)\Delta(x)
\in \Gr_{\bar{\pi}},\quad \forall x\in\frak g.$$
Also, note that ${\rm Im}\,\psi$ generates $\Gr_{\bar{\pi}}$ as an algebra. Hence, $\psi$ induces a Hopf algebra isomorphism $\bar\psi:U(\Gr_\pi)\to \Gr_{\bar{\pi}}$ by the same argument as for $\bar\varphi$.
\end{proof}

\section{Derived actions from crossed homomorphisms}
\mlabel{sec:derived}
In this section, we show that  a crossed homomorphism on either a Lie algebra, a group or a Hopf algebra gives a derived action. Furthermore, these derived structures are related in the same way that the original structures are related.

\subsection{Derived actions for Lie algebras and Lie groups}
\mlabel{ss:derivedlie}
Let $\phi:\g\to\Der(\h)$ be an action of a Lie algebra  $(\g,[\cdot,\cdot]_\g)$  on   a Lie algebra   $(\mathcal{\h},[\cdot,\cdot]_\h)$.
By \cite[Lemma 2.6]{PSTZ}, a crossed homomorphism  $d:\g\to \h$ gives rise to a derived action $\phi_d:\g\to \Der(\h)$ given by
\begin{equation}\mlabel{diff-alg-action}
	\phi_d(x)(u)=\phi(x)(u)+[d(x),u]_\h,\quad \forall x\in\g, u\in\h.
\end{equation}
Furthermore, $-d$ is a Lie algebra crossed homomorphism with respect to the derived action $\phi_d$.

For group crossed homomorphisms, we have a similar result.
\begin{lemma}\mlabel{lem:da}
	Let  $D:\gG\to \gH$ be a group crossed homomorphism on $\gG$ with respect to an action  $(\gH,\Phi)$. There is a derived action $\Phi_D:\gG\to \Aut(\gH)$ by automorphisms given by
	\begin{equation}\mlabel{diff-gp-action}
		\Phi_D(g)h=\Ad_{D(g)}\Phi(g)(h)=D(g)\Phi(g)(h)D(g)^{-1},\quad \forall g,h\in \gG.
	\end{equation}
Moreover, the map
$$\overline{D}:\gG\to \gH, \quad \overline{D}(g):=D(g)^{-1}, \quad \forall g\in \gG$$
is a  group crossed homomorphism, called the {\bf derived crossed homomorphism} of $D$ with respect to the derived action  $(\gH,\Phi_D)$.
\end{lemma}

\begin{proof}
	First since $\Phi(g)$ and $\Ad_{D(g)}$ are automorphisms, it is obvious that $\Phi_D(g)$ is an automorphism of $\gH$ for all $g\in \gG.$
Also, by Eq.~\meqref{crossed-homo-group}, we have
	\begin{eqnarray*}
		\Phi_D(g_1g_2)(h)&=&\Ad_{D(g_1g_2)}\Phi(g_1g_2)(h)\\
		&=&\Ad_{D(g_1)\Phi(g_1)(D(g_2))}\Phi(g_1)\Phi(g_2)(h)\\
		&=&\Ad_{D(g_1)}\Ad_{\Phi(g_1)(D(g_2))}\Phi(g_1)\Phi(g_2)(h)\\
		&=&\Ad_{D(g_1)}\Phi(g_1)\Ad_{D(g_2)}\Phi(g_2)(h)\\
		&=&\Phi_D(g_1)\Phi_D(g_2)(h).
	\end{eqnarray*}
	Thus $\Phi_D:\gG\to \Aut(\gH)$ is an action by automorphisms.

Furthermore, for all $g,h\in \gG$, we have
\begin{eqnarray*}
 \overline{D}(g)\Phi_D(g)( \overline{D}(h))&=&D(g)^{-1}D(g)\Phi(g)( {D}(h)^{-1})D(g)^{-1}\\
 &=&(\Phi(g)( {D}(h)))^{-1}D(g)^{-1}\\
 &=&(D(g)\Phi(g)( {D}(h)))^{-1}\\
 &=&  \overline{D}(gh),
\end{eqnarray*}
which implies that $\overline{D}:\gG\to \gH$    is a   crossed homomorphism on $\gG$ with respect to the action  $(\gH,\Phi_D)$.
\end{proof}

We next establish a relationship between the induced action on a Lie algebra given by Eq.~\meqref{diff-alg-action} and the one on a Lie group given by Eq.~\meqref{diff-gp-action}. Here a crossed homomorphism on a Lie group $\gG$ with respect to an action $\Phi:\gG\to \Aut(\gH)$ is a smooth map $D:\gG\to \gH$ such that \meqref{crossed-homo-group} holds.

Let $(\frak g,[\cdot,\cdot]_\g)$  be the Lie algebra of the Lie group $\gG$.  Let $\exp:\g\to \gG$ denote the exponential map. Then the relation between the Lie bracket $[\cdot,\cdot]_\g$ and the Lie group multiplication is given by the following fundamental formula:
\begin{equation}\mlabel{eq:expo}
	[u,v]_\g=\frac{d^2}{dt ds}\,\bigg|_{t,s=0}\exp(tu)\exp(sv)\exp(-tu),\quad \forall u,v\in\g.
\end{equation}
For any $g\in \gG$, since $\Phi_D(g)$ is in $\Aut(\gH)$, it follows that $\Phi_D(g)(e_\gH)=e_\gH$ for the unit element $e_\gH$ in the Lie group $\gH$. By taking the differentiation, $\Phi_D(g)$ induces $\Phi_D(g)_*\in\Aut(\h)$. Consequently, we obtain a Lie group homomorphism, which is still denoted by $\Phi_D$, from $\gG$ to $\Aut(\h)$. Again taking the differentiation,   we obtain a Lie algebra homomorphism $(\Phi_D)_*:\g\to \Der(\h)$. The above process can be summarized in the following diagram:
\begin{equation}\mlabel{eq:relation}
	\small{
		\xymatrix{\gG \ar[rr]^{\Phi_D}\ar[d]_{\text{differentiation}} &  &  \Aut(\h) \ar[d]^{\text{differentiation}}  \\
			\g \ar[rr]^{ ( \Phi_D)_* } & &\Der(\h).}
	}
\end{equation}

On the other hand, it is obvious that the Lie group action $\Phi:\gG\to \Aut (\gH)$ induces a Lie algebra action of $\phi:\g\to\Der(\h)$, and the relation is given by
$$
\frac{d^2}{dtds}\,\bigg|_{t,s=0}\Phi( \exp_\gG(tx))(\exp_\gH(su)) =\phi(x)(u), \quad \forall x\in\g, u\in \h.
$$
By a discussion similar to the one in \cite[Theorem 2.17]{GLS}, $d:=D_{*e}:\g\to\h$ is a  crossed homomorphism on the Lie algebra $(\frak g,[\cdot,\cdot]_\g)$ with respect to the action $\phi:\g\to \Der(\h)$, and therefore,
$$
\frac{d}{dt}\,\bigg|_{t=0}D(\exp_\gG(tx))=\frac{d}{dt}\,\bigg|_{t=0}\exp_\gH(td(x)),\quad\forall x\in \g.
$$

Now we show that the differentiation of the derived action from a Lie group crossed homomorphism is the derived action from the Lie algebra crossed homomorphism from differentiation.

\begin{theorem}\mlabel{thm:diffcrossed}
	Let $D:\gG\to \gH$ be a crossed homomorphism on a Lie group $\gG$ with respect to an action $\Phi:\gG\to \Aut(\gH)$, and $d=D_{*e}$. Then we have
	$$
	( \Phi_D)_*=\phi_d.
	$$

Moreover, the differentiation of the derived Lie group crossed homomorphism $\overline{D}$ is exactly the derived Lie algebra crossed homomorphism $-d$.
\end{theorem}
\begin{proof}
	By the commutative diagram in Eq.~\meqref{eq:relation}, we have
	\begin{eqnarray*}
		( \Phi_D)_*(x)(u)&=&\frac{d}{dt}\,\bigg|_{t=0}\Phi_D(\exp_\gG(tx))(u)\\
		&=&\frac{d^2}{dtds}\,\bigg|_{t,s=0}\Phi_D(\exp_\gG(tx))(\exp_\gH(su))\\
		&=&\frac{d^2}{dtds}\,\bigg|_{t,s=0} \Ad_{D(\exp_\gG(tx))}\Phi(\exp_\gG(tx))\exp_\gH(su)   \\
		&=&\frac{d^2}{dtds}\,\bigg|_{t,s=0} \Ad_{\exp_\gH(td(x))}\Phi(\exp_\gG(tx))\exp(su) \\
		&=&\frac{d^2}{dtds}\,\bigg|_{t,s=0}\Phi( \exp_\gG(tx))(\exp_\gH(su))
		+\frac{d^2}{dtds}\,\bigg|_{t,s=0} \Ad_{\exp_\gH(td(x))}\exp_\gH(su) \\
		&=&\phi(x)(u)+[d(x),u]_\h,
	\end{eqnarray*}
	which implies that the differentiation of $\Phi_D$ is exactly $\phi_d.$

It is obvious that the differentiation of the derived Lie group crossed homomorphism $\overline{D}$ is exactly the  derived Lie algebra crossed homomorphism $-d$.
\end{proof}

\subsection{Module and module bialgebra characterizations of Hopf algebra crossed homomorphisms}
\mlabel{ss:modch}

Under the cocommutativity condition, we give another  characterization  of Hopf algebra crossed homomorphisms utilizing derived module structures and module bialgebra structures.

\begin{theorem}\mlabel{th:crosshomo-action}
Let $\hK$ and $\hH$ be Hopf algebras such that $\hH$ is an $\hK$-module algebra via an algebra action $\rightharpoonup:\hK\to \End(\hH)$. A coalgebra homomorphism $\pi:\hK\to \hH$ is a Hopf algebra crossed homomorphism if and only if the action
$$a\cdot_\pi x\coloneqq \pi(a_1)(a_2\rightharpoonup x),\quad \forall x\in \hH,\,a\in \hK,$$
defines an $\hK$-module structure on $\hH$.
\end{theorem}
\begin{proof}
If $\pi:\hK\to \hH$ is a crossed homomorphism, then we have
\begin{align*}
ab\cdot_\pi x&= \pi(a_1b_1)(a_2b_2\rightharpoonup x)\\
&=\pi(a_1)(a_2\rightharpoonup \pi(b_1))(a_3\rightharpoonup (b_2\rightharpoonup x))\\
&=\pi(a_1)(a_2\rightharpoonup \pi(b_1)(b_2\rightharpoonup x))
\\&=a\cdot_\pi(b\cdot_\pi x), \quad \forall x\in \hH,\,a,b\in \hK.
\end{align*}
This means that the action $\cdot_\pi$ defines an $\hK$-module structure on $\hH$.

Conversely, if there is the stated $\hK$-module structure $\cdot_\pi$ on $\hH$, then
\begin{align*}
\pi(ab)&=\pi(a_1b_1)(a_2b_2\rightharpoonup 1)=ab\cdot_\pi 1= a\cdot_\pi(b\cdot_\pi 1)\\
&=\pi(a_1)(a_2\rightharpoonup \pi(b_1)(b_2\rightharpoonup 1))
=\pi(a_1)(a_2\rightharpoonup \pi(b_1)\vep_\hK(b_2))\\&=\pi(a_1)(a_2\rightharpoonup \pi(b)),
\quad \forall a,b\in \hK.
\end{align*}
Hence $\pi:\hK\to \hH$ is a crossed homomorphism.
\end{proof}

Note that in the above theorem, even though $\hH$ is an $\hK$-module via $\cdot_\pi$, it is in general not an $\hK$-module algebra action. In the sequel, we will give a new $\hK$-module algebra structure under the cocommutativity condition of $\hK$. The following result is straightforward to check.

\begin{lemma}\mlabel{lemma:smash-modalg}
Let $\hH$ and $\hK$ be Hopf algebras such that $\hK$ is cocommutative and $\hH$ is an $\hK$-module bialgebra via an action $\rightharpoonup$. Then $\hH$ is a $\hH\#\hK$-module algebra with the action defined by
\begin{equation*}
(x\#a).y\coloneqq \ad_x(a\rightharpoonup y),\quad \forall x,y\in \hH,\,a\in \hK.
\end{equation*}
\end{lemma}

\begin{theorem}\mlabel{th:new-action}
Let $\hH$ and $\hK$ be Hopf algebras with $\hK$ cocommutative. Suppose that $\hH$ is an $\hK$-module bialgebra via an action $\rightharpoonup$. Let $\pi:\hK\to \hH$ be a Hopf algebra crossed homomorphism. Then $\hH$ has another $\hK$-module bialgebra structure via the derived action given by
\begin{equation}\mlabel{new-action}
a \rightharpoonup_\pi x\coloneqq \ad_{\pi(a_1)}(a_2\rightharpoonup x),\quad \forall x\in \hH,\,a\in \hK.
\end{equation}
Moreover, $S_\hH\,\pi:\hK\to \hH$ is a Hopf algebra crossed homomorphism with respect to the new action $\rightharpoonup_\pi$, called the {\bf derived crossed homomorphism} of $\pi$.
\end{theorem}
\begin{proof}
By Corollary~\mref{coro:leftgr}, the graph $\Gr_{\pi}$ is a Hopf subalgebra of
$\hH\#\hK$. Thanks to Lemma~\mref{lemma:smash-modalg} and the cocommutativity of $\hK$, $\hH$ becomes a $\Gr_{\pi}$-module bialgebra.
Pulled back by the Hopf algebra isomorphism $\Psi:\hK\to \Gr_{\pi}$ given in Eq.~\meqref{eq:isom-leftgr}, $\hH$ becomes an $\hK$-module bialgebra via the desired action $\rightharpoonup_\pi$.

On the other hand, for any $a,b\in \hK$,
\begin{align*}
S_\hH(\pi(ab))&=S_\hH(\pi(a_1)(a_2\rightharpoonup \pi(b)))\\
&=(a_1\rightharpoonup S_\hH(\pi(b)))S_\hH(\pi(a_2))\\
&=S_\hH(\pi(a_1))\ad_{\pi(a_2)}(a_3\rightharpoonup S_\hH(\pi(b)))\\
&=S_\hH(\pi(a_1))(a_2\rightharpoonup_\pi S_\hH(\pi(b))).
\end{align*}
Hence, $S_\hH\circ\pi:\hK\to \hH$ is a crossed homomorphism with respect to the action $\rightharpoonup_\pi$.
\end{proof}

From Theorem~\mref{th:new-action} we evidently have
\begin{coro}
 The above derived $\hK$-module bialgebra structure on $\hH$ is consistent with the derived actions on the corresponding Lie algebras and groups as follows.
 \begin{enumerate}
   \item For $x\in P(\hH)$ and $a\in P(\hK)$, we have $a \rightharpoonup_\pi x= [\pi(a),y]+(a\rightharpoonup x)$, which is exactly the derived action given by Eq. \meqref{diff-alg-action} (see also \cite[Lemma 2.6]{PSTZ}).

   \item For $x\in G(\hH)$ and $a\in G(\hK)$, we have $a \rightharpoonup_\pi x=\pi(a)(a\rightharpoonup x)\pi(a)^{-1}$, which is exactly the derived action given by Eq.~\meqref{diff-gp-action}.
 \end{enumerate}
 \mlabel{co:derived}
\end{coro}

By Proposition \mref{prop:restrict-cross}, the restriction $(S_\hH\,\pi)|_{G(\hK)}$ of the derived crossed homomorphism is a crossed homomorphism on the group $G(\hK)$ with respect to the restriction of the derived action $\rightharpoonup_\pi$, which is exactly the derived crossed homomorphism $\overline{\pi|_{G(\hK)}}$ of the restriction $\pi|_{G(\hK)}$ given in Lemma \mref{lem:da}.

 \section{Structure theorem of pointed cocommutative difference Hopf algebras}
\mlabel{sec:pcdha}
In this section, we first give some characterizations of \difcops on a Hopf algebra. We then study \difcops on smash product Hopf algebras. In particular, we give the structure theorem of pointed cocommutative difference Hopf algebras.

\begin{lemma}
   A coalgebra homomorphism $D:H\to H$ is a   \difcop on $H$  if and only if
\begin{equation}\mlabel{eq:diff-Hopf'}
D(x_1y)x_2=D(x_1)x_2D(y),\quad \forall\,x,y\in H.
\end{equation}
 \end{lemma}
 \begin{proof}
 If Eq.~\meqref{eq:diff-Hopf} holds, then $D(x_1y)x_2=D(x_1)x_2D(y)S(x_3)x_4=D(x_1)x_2D(y)$. Conversely,
if Eq.~\meqref{eq:diff-Hopf'} holds, then $D(xy)=D(x_1y)x_2S(x_3)=D(x_1)x_2D(y)S(x_3)$.
 \end{proof}

\begin{theorem}\mlabel{th:diff-conv}
Let $H$ be a Hopf algebra and  $D:H\longrightarrow H$ a  coalgebra homomorphism.   Then $D$ is a \difcop on $H$ if and only if
the convolution product $D\con\id$ is an algebra homomorphism.
\end{theorem}
\begin{proof}
For any $x,y\in H$, we have
\begin{align*}
&(D\con\id)(xy)=D((xy)_1)(xy)_2=D(x_1y_1)x_2y_2,\\
&(D\con\id)(x)(D\con\id)(y)=D(x_1)x_2D(y_1)y_2.
\end{align*}
Hence, by Eq.~\meqref{eq:diff-Hopf'}, we find that
$$(D\con\id)(xy)=(D\con\id)(x)(D\con\id)(y),\quad \forall x,y\in H,$$
is equivalent to  Eq.~\meqref{eq:diff-Hopf}, giving the desired equivalence.
\end{proof}

\begin{coro}\mlabel{coro:lie-gp-conv}
There exists a set bijection
between $\Dif(\gG)$ of \difcops and $\End(\gG)$ of group endomorphisms for every group $\gG$,
$$\begin{array}{ccc}
\Dif(\gG) & \longrightarrow  & \End(\gG) \\
D & \longmapsto   & (g\mapsto D(g)g)
\end{array},\quad
\begin{array}{ccc}
\End(\gG)& \longrightarrow & \Dif(\gG)\\
F & \longmapsto  & (g \mapsto F(g)g^{-1})
\end{array}.$$
and a set bijection between $\Dif(\frak g)$ of \difcops and $\End(\frak g)$ of Lie algebra endomorphisms
for every Lie algebra $\frak g$,
$$\begin{array}{ccc}
\Dif(\frak g) & \longrightarrow  & \End(\frak g) \\
d & \longmapsto   & (x\mapsto d(x)+x)
\end{array},\quad
\begin{array}{ccc}
\End(\frak g)& \longrightarrow & \Dif(\frak g)\\
f & \longmapsto  & (x \mapsto f(x)-x)
\end{array}.$$
\end{coro}

Denote by $\Hop(H)$ the monoid of Hopf algebra endomorphisms of $H$ and $\Aut(H)$ the group of Hopf algebra automorphisms of $H$, and denote by ${\Dif}(H)$ the  set   of \difcops on $H$.

\begin{coro}\mlabel{coro:diff-conv}
Let $H$ be a cocommutative Hopf algebra. We have the following set bijection between $\Dif(H)$ of \difcops and $\Hop(H)$ of Hopf algebra endomorphisms,
$$\begin{array}{ccc}
\Dif(H) & \longrightarrow  & \Hop(H) \\
D & \longmapsto   & D\con\id
\end{array},\quad
\begin{array}{ccc}
\Hop(H)& \longrightarrow & \Dif(H)\\
F & \longmapsto  & F\con S
\end{array}.$$
Via such a bijection, we can define a monoid structure {$(\Dif(H),\star,u\circ\vep)$} by
$$D\star D'\coloneqq((D\con\id)D')\con D = (D(D'\con\id))\con D',\quad\forall D,D'\in\Dif(H).$$
\end{coro}
\begin{proof}
When $H$ is cocommutative, $D\con\id$ is also a coalgebra homomorphism for $D\in\Dif(H)$.  By Theorem \mref{th:diff-conv}, $D\con\id$ is a Hopf algebra homomorphism. Then
the desired bijection follows from the fact that the antipode $S$ is the convolution inverse of the identity map $\id$.

For any $D,D'\in\Dif(H)$, applying this bijection we obtain
\begin{align*}
D\star D'&= ((D\con\id)(D'\con\id))\con S\\
&= (((D\con\id)D')\con (D\con\id))\con S\\
&= (((D\con\id)D')\con D)\con (\id\con S)\\
&= ((D\con\id)D')\con D,
\end{align*}
or
\begin{align*}
D\star D'&= ((D\con\id)(D'\con\id))\con S\\
&= ((D(D'\con\id))\con (D'\con\id))\con S\\
&= ((D(D'\con\id))\con D')\con (\id\con S)\\
&= (D(D'\con\id))\con D',
\end{align*}
and $u\circ \vep$ serves as the unit of this monoid.
\end{proof}

\begin{prop}\mlabel{prop:aut-diff}
For a Hopf algebra $H$, $\Aut(H)$ acts on $\Dif(H)$ by conjugation.
\end{prop}
\begin{proof}
Since $\End(H)$ is an $\Aut(H)$-module by conjugation, it is enough to show that
$$\sigma D\sigma^{-1}\in \Dif(H),\quad\forall \sigma\in\Aut(H),D\in \Dif(H).$$
Indeed, $\sigma D\sigma^{-1}$ is still a coalgebra homomorphism, and
$$(\sigma D\sigma^{-1})\con\id=\sigma(D\con\id)\sigma^{-1}$$
is an algebra homomorphism by Theorem~\mref{th:diff-conv}. Thus $\sigma D\sigma^{-1}$ is in $\Dif(H)$.
\end{proof}

The characterization of \difcops in Theorem \mref{th:diff-conv}
inspires us to give the following notion of \difc module bialgebras over \difchopfs.
\begin{defn}
Let $(\hH,D_\hH)$ and $(\hK,D_\hK)$ be cocommutative \difchopfs. The pair $(\hH,D_\hH)$ is called a {\bf \difc $(\hK,D_\hK)$-module bialgebra} via an action $\rightharpoonup:\hK\to \End(\hH)$, if $K$ is an $\hK$-module bialgebra via $\rightharpoonup$ such that
\begin{equation}\mlabel{eq:diff-mod-bialg}
(D_K\con\id)(a\rightharpoonup x)=(D_\hK\con\id)a\rightharpoonup (D_K\con\id)x,\quad\forall x\in \hH,\, a\in \hK.
\end{equation}
Namely, the following equality holds.
\begin{equation}\mlabel{eq:smash-diff}
D_K(a_1\rightharpoonup x_1)(a_2\rightharpoonup x_2)=D_\hK(a_1)a_2\rightharpoonup D_K(x_1)x_2,\quad\forall x\in \hH,\, a\in \hK.
\end{equation}
\end{defn}

\begin{exam}
(i) Any cocommutative \difchopf $(H,D_H)$ is a \difc module bialgebra over itself via the adjoint action.

(ii) Let  $(\gG,D_\gG)$ be a \difc group and $(\frak g,D_\g)$ a \difc Lie algebra. Let    $\rightharpoonup$ be an action  of $\gG$ on $\frak g$ as an automorphism group of Lie algebras. Then $U(\frak g)$ endowed with the extended \difcop, which is also denoted by  $D_\g$, is a difference Hopf algebra, and $\bk \gG$ endowed with the extended \difcop, which is also denoted by  $D_\gG$, is a difference Hopf algebra. Moreover, $(U(\g),D_\g)$ becomes a \difc $(\bk \gG,D_\gG)$-module bialgebra, if and only if
\begin{equation*}
(D_{\frak g}+\id)(g\rightharpoonup x)=D_\gG(g)g\rightharpoonup(D_{\frak g}+\id)x,\quad\forall x\in \frak g,\, g\in \gG.
\end{equation*}
Indeed, it is straightforward to verify Eq.~\meqref{eq:smash-diff} here by induction on the degree of any $u\in U(\frak g)$.
\end{exam}

Next we study \difcops on the smash product of Hopf algebras, and give the following extension theorem of \difcops on cocommutative Hopf algebras.

\begin{theorem}\mlabel{th:smash-diff}
Let $(\hK,D_\hK)$ and $(\hH,D_\hH)$ be cocommutative \difchopfs such that $\hH$ is an $\hK$-module bialgebra via an action $\rightharpoonup$. The smash product Hopf algebra $\hH\#\hK$ has the unique \difcop $D$ such that $D|_\hH=D_\hH$ and $D|_\hK=D_\hK$ if and only if
$(\hH,D_\hH)$ is a \difc $(\hK,D_\hK)$-module bialgebra.
\end{theorem}

\begin{proof}
First suppose that the pair $(D_\hH,D_\hK)$ can be extended to a \difcop $D$ on $\hH\#\hK$. By Corollary~\mref{coro:diff-conv} $D_\hH\con\id$, $D_\hK\con\id$ and $D\con\id$ are all Hopf algebra homomorphisms. Then for any $x\in \hH,\, a\in \hK$,
\begin{align*}
(D\con\id)((1\#a)(x\#1))&=(D\con\id)((a_1\rightharpoonup x)\#a_2)\\
&=(D\con\id)(((a_1\rightharpoonup x)\#1)(1\#a_2))\\
&=(D\con\id)((a_1\rightharpoonup x)\#1)(D\con\id)(1\#a_2)\\
&=((D_\hH\con\id)(a_1\rightharpoonup x)\#1)(1\#(D_\hK\con\id)a_2)\\
&=(D_\hH\con\id)(a_1\rightharpoonup x)\#(D_\hK\con\id)a_2.
\end{align*}
On the other hand,
\begin{align*}
(D\con\id)((1\#a)(x\#1))&=(D\con\id)(1\#a)(D\con\id)(x\#1)\\
&=(1\#(D_\hK\con\id)a)((D_\hH\con\id)x\#1)\\
&=((D_\hK\con\id)a_1\rightharpoonup (D_\hH\con\id)x) \# (D_\hK\con\id)a_2.
\end{align*}
Applying $\id\otimes\vep_\hK$ to both equalities, we conclude that Eq.~\meqref{eq:diff-mod-bialg} holds, and $(\hH,D_\hH)$ is a \difc $\hK$-module bialgebra.

\smallskip
Conversely, assume that Eq.~\meqref{eq:diff-mod-bialg} holds. We define a linear operator $D$ on $\hH\#\hK$ by
$$ D(x\#a)\coloneqq D_\hH(x_1)x_2(D_\hK(a_1)\rightharpoonup S_\hH(x_3))\# D_\hK(a_2),\quad\forall x\in \hH,\, a\in \hK.$$
Then clearly $D|_\hH=D_\hH$ and $D|_\hK=D_\hK$, and we have
\begin{align*}
(D\con\id)(x\#a) &= D(x_1\# a_1)(x_2\# a_2)\\
&=\big(D_\hH(x_1)x_2(D_\hK(a_1)\rightharpoonup S_\hH(x_3))\# D_\hK(a_2)\big)(x_4\# a_3)\\
&=D_\hH(x_1)x_2(D_\hK(a_1)\rightharpoonup S_\hH(x_3))(D_\hK(a_2)\rightharpoonup x_4)\# D_\hK(a_3)a_4\\
&=D_\hH(x_1)x_2(D_\hK(a_1)\rightharpoonup S_\hH(x_3)x_4) \# D_\hK(a_2)a_3\\
&=D_\hH(x_1)x_2 \# D_\hK(a_1)a_2\\
&=(D_\hH\con\id)x\#(D_\hK\con\id)a.
\end{align*}
Also, since both $D_\hH$ and $D_\hK$ are coalgebra homomorphisms, $\hH$ is a cocommutative $\hK$-module bialgebra and $\hK$ is also cocommutative, we conclude that $D$ is a coalgebra homomorphism.

For $x,y\in \hH$ and $a,b\in \hK$, we have
\begin{align*}
(D\con\id)((x\#a)(y\#b))&=(D\con\id)(x(a_1\rightharpoonup y)\#a_2b)\\
&=(D_\hH\con\id)(x(a_1\rightharpoonup y))\# (D_\hK\con\id)(a_2b)\\
&=(D_\hH\con\id)x(D_\hH\con\id)(a_1\rightharpoonup y)\# (D_\hK\con\id)a_2(D_\hK\con\id)b\\
&=(D_\hH\con\id)x((D_\hK\con\id)a_1\rightharpoonup (D_\hH\con\id)y)\# (D_\hK\con\id)a_2(D_\hK\con\id)b\\
&=\big((D_\hH\con\id)x\# (D_\hK\con\id)a\big) \big((D_\hH\con\id)y\#(D_\hK\con\id)b\big)\\
&=(D\con\id)(x\#a)(D\con\id)(y\#b),
\end{align*}
where the fourth equality uses Eq.~\meqref{eq:diff-mod-bialg}.
Hence, $D\con\id$ is an algebra homomorphism, and $D$ is a \difcop on $\hH\#\hK$ such that $D|_\hH=D_\hH$ and $D|_\hK=D_\hK$ by Theorem~\mref{th:diff-conv}.

On the other hand, $D(x\#a)=D((x\#1)(1\#a))$ is determined by $D(x\#1)=D_\hH(x)\#1$ and $D(1\#a)=1\#D_\hK(a)$ via Eq.~\meqref{eq:diff-Hopf}, showing that such extension of \difcops is unique.
\end{proof}

\begin{remark}
In the context of Theorem~\mref{th:smash-diff},
when $\hK$ acts trivially on $\hH$, the compatibility condition in Eq.~\meqref{eq:diff-mod-bialg} holds automatically. Then the tensor product $(\hH\otimes \hK, D_\hH\otimes D_\hK)$ of \difchopfs is also a \difchopf.
\end{remark}

Let $\Phi:\gG\to\Aut(\gH)$ be an action of a group $\gG$  on  $\gH$. Then $\Phi$ can be linearly extended to a module bialgebra action  $\bar\Phi:\bk \gG\to \End(\bk \gH)$ by
$$
\bar\Phi\Big(\sum_{g\in \gG}a_g g\Big)\Big(\sum_{h\in \gH}b_h h\Big)=\sum_{g\in \gG, h\in \gH}a_gb_h \Phi(g)(h).
$$
Likewise, a group crossed homomorphism can be linearly extended to a Hopf algebra crossed homomorphism on the group Hopf algebras.

By Theorem~\mref{th:smash-diff}, for cocommutative \difchopfs $(\hH,D_\hH)$ and $(\hK,D_\hK)$ such that $(\hH,D_\hH)$ is a \difc $(\hK,D_\hK)$-module bialgebra, one can define
the {\bf smash product \difchopf} $(\hH\#\hK,D)$ such that $D|_\hH=D_\hH$ and $D|_\hK=D_\hK$.

On the other hand, the structure theorem of pointed cocommutative Hopf algebras, called the {\bf Cartier-Kostant-Milnor-Moore theorem}, states that such a Hopf algebra $H$ is isomorphic to the smash product Hopf algebra $U(P(H))\#\bk G(H)$
of a universal enveloping algebra and a group algebra; see e.g.~\cite[Theorem~15.3.4]{Ra}.

Next we strengthen this theorem to a structure theorem of pointed cocommutative \difchopfs.

\begin{theorem} \mlabel{thm:dckmmt}
	$(${\bf Difference Cartier-Kostant-Milnor-Moore Theorem}$)$
A pointed cocommutative \difchopf $(H,D)$ is isomorphic to
the smash product \difchopf $U(P(H))\#\bk G(H)$, where
$U(P(H))$ is the \difc $\bk G(H)$-module bialgebra induced by the conjugation action of $G(H)$ on $P(H)$.

\end{theorem}
\begin{proof}
Let $\gG\coloneqq G(H)$ with $D_\gG\coloneqq D|_\gG$, and $\frakg\coloneqq P(H)$ with $D_\frakg\coloneqq D|_\frakg$. By Corollary~\mref{coro:restrict-diff},  $(\gG,D_\gG)$ is a \difc group and $(\frakg,D_\frakg)$ is a \difc Lie algebra.

By Proposition \mref{prop:cross-uea}, the \difcop $D_{\frak g}$ of $\frak g$ can be uniquely extended to $U(\frak g)$. On the other hand, the \difcop $D_\gG$ on $\gG$ is linearly extended to $\bk \gG$. We will use the same notation $D_{\frak g}$ (resp. $D_\gG$) for such extended \difcop on $U(\frak g)$ (resp. $\bk \gG$).

By the Cartier-Kostant-Milnor-Moore theorem, there exists a Hopf algebra isomorphism
$$\Phi:U(\frak g)\#\bk \gG\to H,$$
where $U(\frak g)$ is the $\bk \gG$-module bialgebra induced by the conjugation action of $\gG$ on $\frak g$. We denote the extended action of $\gG$ on $U(\frak g)$ by $\rightharpoonup$.

By Corollary~\mref{coro:mm-difc}, we have $\Phi D_{\frak g}=D\Phi|_{U(\g)}$. Also, it is clear that $\Phi D_\gG=D \Phi|_{\bk \gG}$. Then
\begin{align*}
(D_{\frak g}\con\id)(g\rightharpoonup u)
&=\Phi^{-1}(D\con\id)\Phi(g\rightharpoonup u)\\
&=\Phi^{-1}(D\con\id)\big(\Phi(g)\Phi(u)\Phi(g)^{-1}\big)\\
&=\Phi^{-1}\big(\Phi(D_\gG\con\id)g \Phi(D_{\frak g}\con\id)u (\Phi(D_\gG\con\id)g)^{-1}\big)\\
&=\Phi^{-1}\big(\Phi((D_\gG\con\id)g\rightharpoonup (D_{\frak g}\con\id)u)\big)\\
&=(D_\gG\con\id)g\rightharpoonup (D_{\frak g}\con\id)u,
\end{align*}
for any $u\in U(\frak g)$ and $g\in \gG$.   Here we identify $U(\frak g)$ and $\bk \gG$ as subalgebras of $U(\frak g)\#\bk \gG$.
Hence, Eq.~\meqref{eq:diff-mod-bialg} holds, and $(U(\frak g),D_{\frak g})$ is a \difc $(\bk \gG,D_\gG)$-module bialgebra.

Therefore, we have the smash product \difchopf $(U(\frak g)\#\bk \gG,D_\#)$, where
$$D_\#(u\#g):=D_{\frak g}(u_1)u_2(D_\gG(g)\rightharpoonup S_{\frak g}(u_3))\# D_\gG(g),\quad\forall u\in U(\frak g),\, g\in \gG.$$
Now it remains to verify that $\Phi D_\#=D\Phi$. Indeed,
\begin{align*}
\Phi(D_\#(u\#g))&=\Phi(D_{\frak g}(u_1)u_2(D_\gG(g)\rightharpoonup S_{\frak g}(u_3)))\Phi(D_\gG(g))\\
&=\Phi(D_{\frak g}(u_1))\Phi(u_2)\Phi(D_\gG(g)\rightharpoonup S_{\frak g}(u_3))\Phi(D_\gG(g))\\
&=D(\Phi(u_1))\Phi(u_2)D(\Phi(g))S(\Phi(u_3))D(\Phi(g))^{-1}D(\Phi(g))\\
&=D(\Phi(u_1))\Phi(u_2)D(\Phi(g))S(\Phi(u_3))\\
&=D(\Phi(u)\Phi(g))=D(\Phi(u\#g)).
\end{align*}
Hence, $\Phi$ gives the desired \difchopf isomorphism.
\end{proof}

\section{Examples of difference operators on Hopf algebras}
\mlabel{sec:exam}
We end the paper with classifying difference operators on several Hopf algebras.
\begin{exam}
Consider the \difcops on the tensor Hopf algebra $(TV,\cdot,\Delta^{\co})$ with the coshuffle coproduct $\Delta^{\co}$. Since this is a connected cocommutative Hopf algebra and is the universal enveloping algebra of the free Lie algebra $\Lie(V)=P(TV)$, by Proposition \mref{prop:cross-uea} and Theorem \mref{CMM-Diff-hopf}, we see that \difcops on $(TV,\cdot,\Delta^{\co})$ are  in one-to-one  correspondence with \difcops on the Lie algebra $\Lie(V)=P(TV)$.
By Corollary \ref{coro:lie-gp-conv} (also see ~\mcite{LGG}), a linear map $D:\Lie(V) \to \Lie(V)$ is a \difcop if and only if $\id +D$ is a Lie algebra endomorphism on $\Lie(V)$; while Lie algebra endomorphisms on $\Lie(V)$ are in one-to-one correspondence with $\Hom(V,\Lie(V))$ by the universal property of $\Lie(V)$.
Thus, we conclude that the set of \difcops on the tensor Hopf algebra $TV$ is in bijection with $\Hom(V,\Lie(V))$.
\end{exam}

Next we give \difcops on two basic low-dimensional noncommutative and non-cocommutative Hopf algebras as examples of \difchopfs.

\begin{exam}
We classify difference operators on Sweedler's 4-dimensional Hopf algebra
\[H_4=\bk\langle 1,g,x,gx\,|\, g^2=1,x^2=0,gx=-xg\rangle,\]
with its coalgebra structure and its antipode given by
\[\Delta(g)=g\otimes g,\quad \Delta(x)=x\otimes 1+g\otimes x,\quad \vep(g)=1,\quad \vep(x)=0,\quad S(g)=g,\quad S(x)=-gx.\]
Further $G(H_4)=\{1,g\}$ and $P(H_4)=0$.

Let $D:H_4\to H_4$ be a \difcop. Then it restricts to a \difcop $D$ on the group $G(H_4)=\{1,g\}$.
Thus there are two cases to consider.

(i) Suppose $D(g)=g$. Then as $D$ is a coalgebra homomorphism, we have
$$\Delta(D(x))=D(x)\otimes 1+ g\otimes D(x).$$
That is, $D(x)\in P_{1,g}(H_4)=\bk(1-g)\oplus\bk x$. Thus there exist $\alpha,\beta\in\bk$ such that $D(x)=\alpha(1-g)+\beta x$. Then Eq.~\meqref{eq:diff-Hopf} implies
\begin{align*}
&D(gx)=D(g)gD(x)g^{-1}= g^2(\alpha(1-g)+\beta x)g=\alpha(g-1)+\beta xg,\\
&D(xg)=D(x)D(g)+D(g)xD(g)+D(g)gD(g)(-gx)=(\alpha(1-g)+\beta x)g-2x.
\end{align*}
So $D(gx)\neq D(-xg)$, a contradiction. Therefore there is no \difcop in this case.

(ii) Suppose $D(g)=1$. Then
$$\Delta(D(x))=D(x)\otimes 1+ 1\otimes D(x).$$
Thus $D(x)$ is in $P(H_4)=0$, and hence $D(x)=0$. Then Eq.~\meqref{eq:diff-Hopf} implies that
\begin{align*}
&D(gx)=D(g)gD(x)g=0,\\
&D(xg)=D(x)D(g)+D(g)xD(g)+D(g)gD(g)(-gx)=x-g^2x=0.
\end{align*}
Thus $D$ defines a \difcop.

In summary, $H_4$ only has one \difcop $D=u\circ\vep$, namely $D(g)=1$ and $D(x)=0$.
\end{exam}

\begin{exam}
We finally determine bijective \difcops on the Kac-Paljutkin Hopf algebra $H_8$, the noncommutative and non-cocommutative semisimple
Hopf algebra of dimension $8$ which has been widely studied;~see e.g. \mcite{Ma,SV,Shi}.

A basis for $H_8$ is given by $\{1, x, y, xy, z, xz,yz, xyz\}$ with the relations
\vspace{-.2cm}
$$x^{2} =y^{2}= 1,\quad z^{2} = \frac{1}{2}(1+x+y-xy), \quad xy = yx,\quad zx=yz,\quad zy=xz.$$
The coalgebra structure and the antipode are defined by
\vspace{-.2cm}
$$
\Delta(x) = x \otimes x, \quad \Delta(y) = y \otimes y,
\quad \Delta(z) =\frac{1}{2} (1 \otimes 1 + 1 \otimes x+ y \otimes 1- y \otimes x )(z\otimes z),$$
$$
\varepsilon(x) = \varepsilon(y) = \varepsilon(z) = 1, \quad S(x) = x, \quad S(y) = y, \quad S(z) = z.
$$
In particular, $G(H_8)=\{1,x,y,xy\}$ and $C_4\coloneqq{\rm span}_\bk\{z, xz,yz, xyz\}$ is the unique simple subcoalgebra of $H_8$ of dimension $>1$.

\smallskip
Now suppose that $D:H_8\to H_8$ is a bijective \difcop, then $D|_{G(H_8)}$ is an abelian group automorphism, and $D(C_4)= C_4$, since $D(C_4)\neq 0$ as $\vep(D(z))=\vep(z)=1$.

First let $D$ be a \difcop such that $D(x)=x$ and $D(y)=y$.

As $D(z)\in C_4$, we can set $D(z)=pz$ with nonzero $p\in\bk G(H_8)$. For any $g,h\in G(H_8)$,
we obtain that $g\sigma(h)z^2=g\sigma(h)p^2z^2$ by the equality $D((gz)(hz))=D((gz)_1)\ad_{(gz)_2}D(hz)$, where $\sigma$ denotes the linear operator on $\bk G(H_8)$ defined by $\sigma(1)=1,\sigma(x)=y,\sigma(y)=x,\sigma(xy)=xy$.
Hence, we need $p^2z^2=z^2$, which implies $p^2=1$ with further computations.
There are 16 solutions of $p$ as listed below:
\vspace{-.1cm}
$$\pm 1,\,\pm x,\,\pm y,\,\pm xy,\,\pm\frac{1}{2}(1+x+y-xy),\,\pm\frac{1}{2}(1+x-y+xy),\, \pm\frac{1}{2}(1-x+y+xy),\,\pm\frac{1}{2}(-1+x+y+xy).$$
On the other hand, we require
\begin{align*}
\Delta D(z)&=(D\otimes D)\Delta(z)\\
&=\frac{1}{2} (D(z) \otimes D(z) + D(z) \otimes D(xz)+ D(yz) \otimes D(z)- D(yz) \otimes D(xz))\\
&=\frac{1}{2} (1 \otimes 1 + 1 \otimes y+ x \otimes 1- x \otimes y)(D(z)\otimes D(z)).
\end{align*}
Then the only choices for $D(z)$ are
\vspace{-.1cm}
$$\frac{1}{2}(1+x+y-xy)z,\quad \frac{1}{2}(1+x-y+xy)z,\quad \frac{1}{2}(1-x+y+xy)z,\quad \frac{1}{2}(-1+x+y+xy)z.$$
For every other standard basis element
$gz\in C_4$ with $g\in G(H_8)$, we correspondingly have
$$\Delta D(gz)=\Delta(D(z)g)=((D\otimes D)\Delta(z))(g\otimes g)=(D\otimes D)((g\otimes g)\Delta(z))=(D\otimes D)\Delta(gz).$$
So $D$ is clearly a coalgebra homomorphism.

\smallskip
Now it is straightforward to check that the bijective \difcops $D$ on $H_8$ such that $D(x)=x$ and $D(y)=y$ are  given by
\vspace{-.1cm}
\begin{enumerate}[(1)]
\item\smallskip
$D_1(1,x,y,xy,z,xz,yz,xyz)$

$=(1,x,y,xy,\tfrac{1}{2}(1+x+y-xy)z,\tfrac{1}{2}(1-x+y+xy)z
,\tfrac{1}{2}(1+x-y+xy)z,\tfrac{1}{2}(-1+x+y+xy)z)$,

\item\smallskip
$D_2(1,x,y,xy,z,xz,yz,xyz)$

$=(1,x,y,xy,\tfrac{1}{2}(1+x-y+xy)z,\tfrac{1}{2}(-1+x+y+xy)z
,\tfrac{1}{2}(1+x+y-xy)z,\tfrac{1}{2}(1-x+y+xy)z)$,

\item\smallskip
$D_3(1,x,y,xy,z,xz,yz,xyz)$

$=(1,x,y,xy,\tfrac{1}{2}(1-x+y+xy)z,\tfrac{1}{2}(1+x+y-xy)z
,\tfrac{1}{2}(-1+x+y+xy)z,\tfrac{1}{2}(1+x-y+xy)z)$,

\item\smallskip
$D_4(1,x,y,xy,z,xz,yz,xyz)$

$=(1,x,y,xy,\tfrac{1}{2}(-1+x+y+xy)z,\tfrac{1}{2}(1+x-y+xy)z
,\tfrac{1}{2}(1-x+y+xy)z,\tfrac{1}{2}(1+x+y-xy)z)$,
\end{enumerate}

Next let $D$ be a \difcop such that $D(x)=y$ and $D(y)=x$. By a similar analysis, we find that all the bijective \difcops $D$ on $H_8$ with this property are given by
\begin{enumerate}[(1)]
\item\smallskip
$D_5(1,x,y,xy,z,xz,yz,xyz)=(1,y,x,xy,z,xz,yz,xyz)$,

\item\smallskip
$D_6(1,x,y,xy,z,xz,yz,xyz)=(1,y,x,xy,xz,z,xyz,yz)$,

\item\smallskip
$D_7(1,x,y,xy,z,xz,yz,xyz)=(1,y,x,xy,yz,xyz,z,xz)$,

\item\smallskip
$D_8(1,x,y,xy,z,xz,yz,xyz)=(1,y,x,xy,xyz,yz,xz,z)$.
\end{enumerate}

On the other hand, there is no bijective \difcop $D$ on $H_8$ such that $D(x)=xy$ or $D(y)=xy$.
Hence, we have determined all bijective \difc structures on $H_8$.
\end{exam}
\vspace{-.1cm}

\noindent
{\bf Acknowledgements. } This research is supported by NSFC (Grant No. 11922110, 12071094, 12001228).

\vspace{-.3cm}

\bibliographystyle{amsplain}

\end{document}